\documentclass[11pt]{article}
\makeatletter

\usepackage{amssymb,amsmath,amsthm,latexsym}
\usepackage[mathscr]{eucal} \usepackage{mathrsfs}
\usepackage{color}
\usepackage{cite}

\usepackage[dvipsnames]{xcolor}
\usepackage[linktocpage]{hyperref}
\hypersetup{
  colorlinks   = true, 
  urlcolor     = blue, 
  linkcolor    = Purple, 
  citecolor   = red 
}
\usepackage{comment}
\let\wfs@comment@comment\comment
\let\comment\@undefined

\usepackage{enumitem}
\usepackage{changes}
\let\wfs@changes@comment\comment
\let\comment\@undefined

\newcommand\comment{%
    \ifthenelse{\equal{\@currenvir}{comment}}
    {\wfs@comment@comment}
    {\wfs@changes@comment}%
}
\definechangesauthor[name=Daniele, color=red]{DAN}
\definechangesauthor[name=Giuseppe, color=purple]{GIUS}
\definechangesauthor[name=Alessandro, color=blue]{ALE}

\title{New MRD codes from linear cutting blocking sets}
\usepackage{authblk}
\author[1]{Daniele Bartoli}
\affil[1]{Department of Mathematics and Informatics, University of Perugia, Perugia,  Italy, {\small \texttt{daniele.bartoli@unipg.it}}\vspace*{.3cm}}
\author[2]{Giuseppe Marino}
\affil[2]{Department of Mathematics and Applications ``R. Caccioppoli'', University of Naples ``Federico II'', Napoli, Italy, {\small\texttt{giuseppe.marino@unina.it}}\vspace*{.3cm}}
\author[3]{Alessandro Neri}
\affil[3]{Max-Planck-Institute for Mathematics in the Sciences, Leipzig, Germany, \small{\texttt{alessandro.neri@mis.mpg.de}}}

\date{}

\DeclareMathOperator{\dd}{d}
\DeclareMathOperator{\rk}{rk}

\begin{document}
\maketitle

\theoremstyle{definition}
\newtheorem{theorem}{Theorem}[section]
\newtheorem{lemma}[theorem]{Lemma}
\newtheorem{conj}[theorem]{Conjecture}
\newtheorem{remark}[theorem]{Remark}
\newtheorem{cor}[theorem]{Corollary}
\newtheorem{prop}[theorem]{Proposition}
\newtheorem{defin}[theorem]{Definition}
\newtheorem{result}[theorem]{Result}
\newtheorem{property}[theorem]{Property}
\newtheorem{question}[theorem]{Question}

\makeatother
\newcommand{\Prf}{\noindent{\bf Proof}.\quad }
\renewcommand{\labelenumi}{(\alph{enumi})}


\def\B{\mathbf B}
\def\C{\mathbf C}
\def\Z{\mathbf Z}
\def\Q{\mathbf Q}
\def\W{\mathbf W}
\def\a{\mathbf a}
\def\b{\mathbf b}
\def\c{\mathbf c}
\def\d{\mathbf d}
\def\e{\mathbf e}
\def\l{\mathbf l}
\def\v{\mathbf v}
\def\w{\mathbf w}
\def\x{\mathbf x}
\def\y{\mathbf y}
\def\z{\mathbf z}
\def\t{\mathbf t}
\def\cD{\mathcal D}
\def\cC{\mathcal C}
\def\cH{\mathcal H}
\def\cM{{\mathcal M}}
\def\cK{\mathcal K}
\def\cQ{\mathcal Q}
\def\cU{\mathcal U}
\def\cS{\mathcal S}
\def\cT{\mathcal T}
\def\cR{\mathcal R}
\def\cN{\mathcal N}
\def\cA{\mathcal A}
\def\cF{\mathcal F}
\def\cL{\mathcal L}
\def\cP{\mathcal P}
\def\cG{\mathcal G}
\def\cGD{\mathcal GD}
\def\mC{\mathcal C}
\def\mU{\mathcal U}

\def\PG{{\rm PG}}
\def\GL{{\rm GL}}
\def\GF{{\rm GF}}
\def\Tr{{\rm Tr}}

\def\Pg{PG(5,q)}
\def\pg{PG(3,q^2)}
\def\ppg{PG(3,q)}
\def\HH{{\cal H}(2,q^2)}
\def\F{\mathbb F}
\def\Ft{\mathbb F_{q^t}}
\def\P{\mathbb P}
\def\V{\mathbb V}
\def\bS{\mathbb S}
\def\G{\mathbb G}
\def\E{\mathbb E}
\def\N{\mathbb N}
\def\K{\mathbb K}
\def\D{\mathbb D}
\def\ps@headings{
 \def\@oddhead{\footnotesize\rm\hfill\runningheadodd\hfill\thepage}
 \def\@evenhead{\footnotesize\rm\thepage\hfill\runningheadeven\hfill}
 \def\@oddfoot{}
 \def\@evenfoot{\@oddfoot}
}
\def\cub{\mathscr C}
\def\cO{\mathcal O}
\def\cur{\mathscr L}
\def\Fqm{{\mathbb F}_{q^m}}
\def\Fq3{{\mathbb F}_{q^3}}
\def\fq{{\mathbb F}_{q}}
\def\Fm{{\mathbb F}_{q^m}}

\newcommand{\Fmk}{[n,k]_{q^m/q}}
\newcommand{\Fmkd}{[n,k,d]_{q^m/q}}
\newcommand{\Fmkdd}{[n,k,(d_1,\ldots,d_k)]_{q^m/q}}
\newcommand{\Ukdd}{\mathfrak U(n,k,(d_1,\ldots,d_k))_{q^m/q}}
\newcommand{\Ckdd}{\mathfrak C(n,k,(d_1,\ldots,d_k))_{q^m/q}}

\newcommand{\ale}[1]{{\color{blue}[$\star \star$ {\sf Alessandro: #1}]}}

\newcommand{\giu}[1]{{\color{red}[$\star \star$ {\sf Giuseppe: #1}]}}

\newcommand{\Fms}[3]{[#1,#2]_{q^{#3}/q}}

\newcommand{\Gal}{\mathrm{Gal}}
\newcommand{\supp}{\mathrm{supp}}

\newcommand{\wt}{\mathrm{wt}}

\begin{abstract}
 Minimal rank-metric codes or, equivalently,  linear cutting blocking sets are characterized in terms of the second generalized rank weight, via their connection with  evasiveness properties of the associated $q$-system. Using this result, we provide the first construction of a family of $\F_{q^m}$-linear MRD codes of length $2m$ that are  not obtained as a direct sum of two smaller MRD codes. In addition, such a family has better parameters, since its codes possess generalized rank weights strictly larger   than those of the previously known MRD codes. This shows that not all the MRD codes have the same generalized rank weights, in contrast to what happens in the Hamming metric setting.
\end{abstract}

\bigskip


\section{Introduction}

\textbf{Overview.} Codes in the rank metric have seen a rapid increase of interest in the last decades, due to their application to network coding. However, researchers have been fascinated by these codes not only for this application, but also for their intrinsic mathematical structure. Indeed, rank-metric codes have been studied in connection with many mathematical areas, such as finite semifields, linear sets in finite geometry, tensors, $q$-analogues in combinatorics. 

Rank-metric codes can be  seen either as vector subspaces of the space of $n \times m$ matrices over a (possibly finite) field, e.g. $\fq$, or as subspaces of vectors of length $n$ over a degree $m$ extension field, e.g. $\Fm$. In both cases, the most important parameters of a rank-metric code are given by $n,m$, its dimension and its minimum rank distance, that is the minimum rank of a nonzero element of the code. These parameters are related by a very elegant inequality, which is known as the Singleton-like bound. Codes attaining this bound with equality are called \emph{maximum rank distance (MRD) codes}, and they are considered to be optimal, due to their largest possible error-correction capability. 

Constructions of MRD codes are known for every parameters if we consider $\fq$-linear rank-metric codes. However, with the further requirement to be $\Fm$-linear, this is not anymore true. Indeed, constructions of $\Fm$-linear MRD codes are known when $n\leq m$ for any possible dimension and finite field size and they were  already proposed in the seminal papers by Delsarte \cite{delsarte1978bilinear} and Gabidulin \cite{gabidulin1985theory}. When $n>m$, the only known constructions concern $3$-dimensional MRD codes with $n=\frac{3m}{2}$, when $m$ is even \cite{bartoli2018maximum,csajbok2017maximum}, and the direct sum of  copies of suitable MRD codes among those just described \cite[Proposition 22]{loidreau}. 

As already well-known for classical coding theory involving the Hamming metric, also rank-metric codes have a geometric equivalent representation. This was initially described in several works \cite{sheekey2016new,csajbok2017maximum,zini2021scattered} for some special cases, and was definitely established independently by Sheekey \cite{sheekey2019scattered} and Randrianarisoa \cite{randrianarisoa2020geometric}. Namely, $\Fm$-linear codes of length $n$ {and }$\Fm$-dimension $k$ can be equivalently represented in a geometric way as $\fq$-subspaces of $\Fm^k$ of $\fq$-dimension $n$. They were named as $q$-systems and provide a sort of dual representation of rank-metric codes, which was shown to be useful in characterizing families of codes, as it happens for codes in the Hamming metric. Already Randrianarisoa in \cite{randrianarisoa2020geometric} exploited this geometric perspective to give a complete classification of $\Fm$-linear one-weight codes in the rank metric. 

Recently, in \cite{ABNR} a new class of rank-metric codes has been introduced and investigated{, the family} of \emph{minimal rank-metric codes}. The peculiarity of these codes is that the set of supports -- which in the rank-metric setting is given by $\fq$-subspaces of $\fq^n$ -- form an antichain with respect to the {inclusion } {and they} are the $q$-analogues of \emph{minimal linear codes} in the Hamming metric, which have been shown to have interesting  combinatorial and geometric properties  and applications to secret sharing schemes, as proposed by Massey {\cite{Massey1993,Massey1995}}. Using the geometric viewpoint, minimal rank-metric codes have been shown in \cite{ABNR} to correspond to linear cutting blocking sets, which are special $q$-systems $U$ with the following intersection property: for every $\Fm$-hyperplane {$H$}, it holds $\langle U\cap H\rangle_{\Fm}=H$.
Furthermore, always in \cite{ABNR}, constructions of minimal rank-metric codes were proposed using \emph{scattered subspaces}  {in a $3$-dimensional $\Fqm$-space.} 

\medskip

\noindent\textbf{Our contribution.} Motivated by the study of minimal rank-metric codes, we make further progresses on linear cutting  blocking sets and their properties. We give a new characterization of them in terms of their evasiveness properties. In particular, in Theorem \ref{Thm:caract} we show that a $q$-system $U$ of $\Fm^k$ of $\fq$-dimension $n$ is a linear cutting blocking set if and only if it is $(k-2,n-m-1)_q$-evasive. This means that every $(k-2)$-dimensional $\Fm$-subspace of $\Fm^k$ intersects $U$ in a space of $\fq$-dimension at most $n-m-1$. In a coding theoretic language, this translates in a very elegant and concise characterization described in Theorem \ref{thm:minimal_secondweight} which can be read as follows:

\begin{center}
  \textbf{Theorem 3.4.}  An $\Fm$-linear (nondegenerate) rank-metric code in $\Fm^n$ is minimal if and only if its second generalized rank weight is strictly greater than the field extension degree $m$.
\end{center}
This result does not seem to have an analogue in the Hamming metric, and thus it is genuinely new. 

Since every $q$-system containing a linear cutting blocking set is itself a linear cutting blocking set, it is then natural to look only for the small ones; or, in other words, for the existence of short minimal rank-metric codes. In this direction, we provide the first answers for small parameters. We completely settle the case $(k,m)=(4,3)$ for every prime power $q$ and then we construct a family of linear cutting blocking sets of $\fq$-dimension $8$ in $\F_{q^4}^4$, when $q$ is an odd power of $2$. Concretely, the latter is proved in Theorem \ref{thm:Uevasive}, where we show that the $q$-system,\begin{equation}\label{eq:uno}U := \left\{\left(x,y,x^q+y^{q^2},x^{q^2}+y^q+y^{q^2}\right) \, : \, x,y \in \mathbb{F}_{q^4} \right\}\end{equation}
is a linear cutting blocking set {of $\F_{q^4}^4$}. 

 As a byproduct, we also show that the codes associated to $U$ give rise to a new family of $4$-dimensional MRD codes. This immediately follows from the geometric result given in Theorem \ref{thm:Uscattered}, where $U$ is proved to be scattered. We then compare these codes with the known constructions of $\Fm$-linear MRD codes, realizing that not only they are new, but they are \emph{structurally new}. Indeed, the codes we construct cannot be obtained as direct sum of smaller MRD codes, in contrast to all the previously known constructions. We show that this also implies that these new codes have \emph{better parameters}, being their second generalized rank weight strictly larger. This is in strong contrast with what happens for the Hamming metric, in which MDS codes have all the same generalized (Hamming) weights. Thus, the new construction provides the first very concrete evidence that \emph{MRD codes can have substantially different parameters}.

\medskip
\noindent\textbf{Outline.} The paper is structured as follows. Section \ref{sec:perliminaries} introduces the main objects and notions that we need in the paper, giving a brief recap on rank-metric codes, $q$-systems and their evasiveness and cutting properties. In Section \ref{sec:cutting_evasive} we explore the link between linear cutting blocking sets and evasive $q$-systems, showing that minimal rank-metric codes can be characterized in terms of their second generalized rank weight. Section \ref{sec:construction_evasive} contains the study of the special $q$-system $U$ in \eqref{eq:uno}, which is shown to be scattered and $(2,3)_q$-evasive with the aid of some technical results. After discussing the  properties of the generalized rank weights of the associated MRD codes, we conclude in Section \ref{sec:conclusions}, listing some open problems and new research directions.

\section{Preliminaries}\label{sec:perliminaries}
In this section we recall the basic notions on $q$-systems, evasive subspaces, linear cutting blocking sets and their relations with linear rank-metric codes. 
We first introduce the setting. Let $q$ be a prime power and let $k,n,m$ be positive integers. We denote by $\fq$ the finite field with $q$ element and by $\Fm$ the extension field of degree $m$. Furthermore, for any positive integer $r$,  $V(r,q^m)$ is a vector space of dimension $r$ over $\Fm$. {This will be always identified with $\Fm^r$. Also, any $1$-dimensional $\Fqm$-vector subspace of $V$ will be called \textbf{point} of $V$.}

\subsection{Rank-metric codes}

Rank-metric codes were originally introduced by Delsarte in the late $70$'s in \cite{delsarte1978bilinear}, for a pure mathematical interest and with no applications in mind. They were reintroduced a few years later by Gabidulin in \cite{gabidulin1985theory}, and afterwards several applications of codes in the rank-metric were proposed, such as crisscross error correction \cite{roth1991maximum}, cryptography \cite{gabidulin1991ideals}, distributed storage \cite{rawat2013optimal}, and network coding \cite{silva2008rank}. Mathematically speaking, one can either define them as set of matrices over a (finite) field, or as set of vectors defined over an extension field. In this work we will only consider the vector representation, and in particular we will focus on rank-metric codes which are linear over the extension field.

On the vector space $V(n,q^m)$ we fix the metric induced by the $\fq$-rank. More precisely, the \textbf{rank weight}  is defined, for $v=(v_1,\ldots,v_n)\in V(n,q^m)$, as
$$ \wt_{\rk}(v)=\dim_{\fq} \langle v_1,\ldots,v_n\rangle_{\fq}.$$
If $\wt_{\rk}(v)=r$ and we fix an $\fq$-basis $u=(u_1,\ldots,u_r)$ of $\langle v_1,\ldots,v_n\rangle_{\fq}$, then there exists a matrix $A\in \fq^{r\times n}$ such that $v=uA$. The \textbf{rank support} of $v$ defined as
$$\supp(v):=\mathrm{rowsp}(A)\subseteq \fq^n,$$
where $\mathrm{rowsp}(A)$ denotes the $\fq$-span of the rows of $A$,
and it is well-defined since it does not depend on the choice of $u$.

The rank weight induces a metric, which is given by the \textbf{rank distance}, defined for $u,v \in V(n,q^m)$ as $\dd_{\rk}(u,v)=\wt_{\rk}(u-v)$.

\begin{defin}
An $\Fmk$ \textbf{(rank-metric) code} is a $k$-dimensional $\Fm$-subspace of $V(n,q^m)$ endowed with the rank distance.
\end{defin}

 Before introducing the metric properties and invariants of a rank-metric code, we introduce the notion of minimality of codewords and of codes; see also \cite{ABNR}. 
\begin{defin}
Let $\mC$ be an $\Fmk$ code. A nonzero codeword $v\in\mC$ is \textbf{minimal} if for every $u \in \mC$,
$$ \supp(u)\subseteq \supp(v)  \Longleftrightarrow u=\lambda v, \mbox{ for some } \lambda \in \Fm.$$
Furthermore, if every nonzero codeword of $\mC$ is minimal, then $\mC$ is said to be a \textbf{minimal rank-metric code}.
\end{defin}

Important invariants of rank-metric codes are given by their generalized rank weights. They have been first introduced and studied by Kurihara \emph{et al.} in  \cite{kurihara2012new,kurihara2015relative}, by Oggier and Sboui in \cite{oggier2012existence} and by Ducoat and Kyureghyan in \cite{ducoat2015generalized}. They are the analogue of  generalized weights in Hamming metric and are of great  interest due to their combinatorial properties (see \cite{ravagnani2016generalized,jurrius2017defining,ghorpade2020polymatroid}) and their applications to network coding (see \cite{martinez2016similarities,martinez2017relative}).

Let $\cA \subseteq V(n,q^m)$ and $\theta \in \Gal(\Fm/\fq)$, and denote by $\theta(\cA)$ the image of $\cA$ under the componentwise map $\theta$, that is 
$$\theta(\cA):= \{(\theta(a_1),\ldots,\theta(a_n)) \, : \, (a_1,\ldots,a_n) \in \cA\}.$$
We say that $\cA$ is \textbf{Galois closed} if  $\theta(\cA)=\cA$ for every $\theta \in \Gal(\Fm/\fq)$. Observe that this is equivalent to require that $\theta(\cA)=\cA$ for a generator of $\Gal(\Fm/\fq)$. For instance, we can just consider $\theta$ to be the $q$-Frobenius automorphism. We denote by $\Lambda_q(n,m)$ the set of Galois closed $\Fm$-subspaces of $V(n,q^m)$.

\begin{defin}
 Let $\mC$ be an $\Fmk$ code, and let $1\leq j \leq k$. The \textbf{$j$-th generalized rank weight} of $\mC$ is the integer
 $$ \dd_{\rk,j}(\mC):=\min \{\dim_{\Fm}(\cA) \,:\, \cA \in  \Lambda_q(n,m), \dim_{\Fm}(\cA\cap \mC) \geq {j}\}.$$
 For brevity, when we want to underline the generalized rank weights of $\mC$ we will call it an $\Fmkdd$ code, where $d_i:=\dd_{\rk,i}(\mC)$. 
\end{defin}

Observe that the first generalize weight coincides with the \textbf{minimum rank distance} of $\mC$, that is
$$ \dd_{\rk,1}(\mC)=\dd(\mC):=\min\{\rk(c) \, :\, c \in \mC\setminus \{0\}\}.$$
When only the minimum rank distance is relevant, we will write that $\mC$ is an $\Fmkd$ code.

The minimum rank distance of a code measures its error correction capability, and hence it is a fundamental parameter and it is crucial to have it as large as possible. However, there are some constraints on the parameters that one has to take into account. The most important one is given by the well-known Singleton bound.

\begin{theorem}[Singleton Bound \textnormal{\cite{delsarte1978bilinear}}]
 Let $\mC$ be an $\Fmkdd$ code. Then
 \begin{equation}\label{eq:singleton} mk\leq \min\{m(n-d_1+1), n(m-d_1+1)\}.\end{equation}
\end{theorem}
A code $\mC$ is said to be \textbf{maximum rank distance (MRD)} {if Bound } \eqref{eq:singleton} is met with equality. 

The first construction of $\Fmk$ MRD codes was provided already by Delsarte \cite{delsarte1978bilinear} and Gabidulin \cite{gabidulin1985theory}, when $n\le m$. These codes are known today as \textbf{Delsarte-Gabidulin codes}. When $n\le m$, another construction was also provided more recently by Sheekey in \cite{sheekey2016new}. 
Additional constructions can be obtained when $n=tm$ and $k=tk'$  via the direct sum of $t$ copies of an $[m,k',m-k'+1]_{q^m/q}$ MRD code, e.g. a Delsarte-Gabidulin code. These constructions have been recently extended using geometric arguments in \cite{bartoli2018maximum,csajbok2017maximum}, for  $k$ odd, $m=2m'$ and $n=m'k$, giving $[km',k,m-1]_{q^m/q}$ MRD codes. They are obtained as direct sum of $[m,2,m-1]_{q^m/q}$ MRD codes (like Delsarte-Gabidulin codes) and $[3m',3,m-1]_{q^m/q}$ MRD codes.
In particular, these are the only constructions of $\Fmkd$ MRD codes with $n>m$.

{Also, }  the generalized rank weights are important for applications and give a measure on the security performance and the error correction capability of secure network coding; see  \cite{kurihara2015relative}. As for the minimum rank distance, one can derive bounds on the parameters of a code involving the generalized rank weights. In this case the bounds are more complicated.

\begin{prop}[Bounds \textnormal{\cite{martinez2016similarities}}]\label{prop:bounds_genweights}
 Let $\mC$ be an  $\Fmkdd$ code. Then for each $s \in \{1,\ldots,k\}$ we have
 \begin{equation}\label{eq:singbound_genweights} d_s \leq \min\Big\{n-k+s, sm, \frac{m}{n}(n-k)+m(s-1)+1 \Big\}.\end{equation}
\end{prop}

Apart from the bounds, the most important properties of the generalized rank weights are given by the strict monotonicity and the Wei-type duality. The latter involves the notion of dual code. If $\mC$ is an $[n,k]_{q^m/q}$ code, we define $\mC^\perp$ to be {its } orthogonal complement with respect to the standard inner product on $V(n,q^m)$. In other words, $\mC^\perp$ is the $[n,n-k]_{q^m/q}$ code given by
$$ \mC^\perp=\{ u \in V(n,q^m) \,:\, uv^\top=0, \mbox{ for every } v \in \mC \}$$
and it is called the \textbf{dual code} of $\mC$.

\begin{prop}[\textnormal{see \cite{kurihara2015relative,ducoat2015generalized}}]\label{prop:gen_weights_properties}
 Let $\mC$ be an $\Fmkdd$ code and let $\mC^\perp$ be its dual $[n,n-k,(d_1^\perp,\ldots,d_{n-k}^\perp)]_{q^m/q}$ code. Then
 \begin{enumerate}[label=(\arabic*)]
     \item $1\leq d_1<d_2<\ldots<d_k\leq n$. \hfill \textbf{(Monotonicity) }
     \item$\{d_1, \ldots,d_k\}\cup \{n+1-d_1^\perp,\ldots,  n+1-d_{n-k}^\perp\}=\{1,\ldots,n\}$. \hfill  \textbf{(Wei-type duality)} 
 \end{enumerate}
\end{prop}

We conclude by recalling the notions of (non)degeneracy and of equivalence of rank-metric codes. There are several equivalent ways to define nondegenerate rank-metric codes; see e.g. \cite[Proposition 3.2]{ABNR}. Here we give the following.

\begin{defin}
An $\Fmkdd$ code is said to be \textbf{nondegenerate} if $\dd_{\rk,k}(\mC)=n$. 
\end{defin}

Also concerning equivalence of codes there are a few ways to introduce this notion. Here we only consider equivalence of codes given by linear isometries of the ambient space $V(n,q^m)$.

\begin{defin}
 Two $\Fmk$ codes $\mC_1, \mC_2$ are said to be \textbf{(linearly) equivalent} if there exists $A\in \GL(n,q)$ such that $\mC_2=\mC_1\cdot A:=\{vA \,:\, v \in \mC_1\}$,
\end{defin}

\subsection{$q$-Systems}

In this section we recall some notions and results on $q$-systems. They were introduced by Sheekey in \cite{sheekey2019scattered} and by Randrianarisoa in \cite{randrianarisoa2020geometric}, as the natural geometric objects describing $\Fmk$ codes. In particular, we will focus on evasive subspaces and linear cutting blocking sets. 

We first introduce the notion of weight with respect to an $\fq$-subspace.
\begin{defin}
Let $U$ be an $\F_q$-subspace of $V(k,q^m)$. For an $\Fm$-subspace $H$ of $V(k,q^m)$, we define the \textbf{weight} of $H$ in $U$ the quantity  $\wt_{U}(H):=\dim_{\fq}(H\cap U)$.
\end{defin}

We now recall the definition of $q$-system, which was given in \cite{randrianarisoa2020geometric} for the first time. 

\begin{defin}
An $\Fmkdd$ \textbf{system} $U$ is an $\fq$-subspace of $V(k,q^m)$, with $\dim_{\fq}(U)=n$ and such that $\langle U \rangle_{\Fm}=V(k,q^m)$. For each $i \in \{1,\ldots,k\}$, the parameter $d_i$ is defined as
\begin{align*} d_i:=&\,n-\max\{\dim_{\fq}(U\cap H) \,:\, H \subseteq V(k,q^m) \mbox{ with } \dim_{\Fqm}(H)=k-i \}\\
=&\,n-\max\{\wt_U(H) \,:\, H \subseteq V(k,q^m) \mbox{ with } \dim_{\Fqm}(H)=k-i \}.
\end{align*}
 When the parameters $(d_1,\ldots,d_k)$ are not relevant/known we will write that $U$ is an $\Fmk$ system. Furthermore, when none of the parameters is relevant, we will generically refer to $U$  as a \textbf{$q$-system}.
 
 Two $\Fmkdd$ systems $U_1,U_2$ are \textbf{(linearly) equivalent} if there exists $A\in \GL(k,q^m)$ such that $ U_1\cdot A:=\{uA \,:\, u \in U_1\}=U_2$.
\end{defin}

The notation and the language used for studying $q$-systems are inherited from the theory of rank-metric codes. This is due to their strong interconnection that was first observed in \cite{sheekey2019scattered}, and  in \cite{randrianarisoa2020geometric}, and then further developed in \cite{ABNR,CMNT}. 

Let $\Ukdd$ denote the set of equivalence classes $[U]$ of $\Fmkdd$ systems, and let $\Ckdd$ denote the set of equivalence classes $[\mC]$ of nondegenerate $\Fmkdd$ codes. One can define the maps
$$\begin{array}{rccc}\Phi: & \Ckdd &\longrightarrow &\Ukdd\\
& [\mathrm{rowsp}( u_1^\top \mid \ldots \mid u_n^\top)] & \longmapsto & [\langle u_1, \ldots u_n\rangle_{\fq}] \end{array}, $$

$$\begin{array}{rccc}\Psi: & \Ukdd &\longrightarrow &\Ckdd \\
& [\langle u_1, \ldots u_n\rangle_{\fq}] & \longmapsto & [\mathrm{rowsp}( u_1^\top \mid \ldots \mid u_n^\top)] \end{array}.$$

\begin{theorem}[see \textnormal{\cite{randrianarisoa2020geometric}}]\label{thm:correspondence_codes_systems}
 The maps $\Phi$ and $\Psi$ are well-defined and they are one the inverse of each other. Hence, they define a one-to-one correspondence between equivalence classes of $\Fmkdd$ codes and equivalence classes of $\Fmkdd$ systems.
\end{theorem}

In light of Theorem \ref{thm:correspondence_codes_systems}, from now on, for any $\Fmkdd$ system $U$, we say that a nondegenerate $\Fmkdd$ code $\mC$ is \textbf{associated} to $U$ if $\mC \in \Psi([U])$. Similarly, an $\Fmkdd$ system $U$ will be said to be \textbf{associated} to a nondegenerate $\Fmkdd$ code $\mC$ if $U\in \Phi([\mC])$.

\bigskip 

A special family of $q$-systems is given by the so-called evasive subspaces, which generalize scattered subspaces.

\begin{defin}
Let $h,r$ be positive integers such that $h<k$. An $\Fmk$ system $U$  is said to be an \textbf{$(h,r)_q$-evasive subspace} (or simply \textbf{$(h,r)_q$-evasive}) if $\dim_{\fq}(U\cap H)\leq r$ for each $\Fm$-subspace $H$ of $V(k,q^m)$ with $\dim_{\Fm}(H)=h$.
 When $r=h$, an $(h,h)_q$-evasive subspace is called \textbf{$h$-scattered}. Furthermore, when $h=1$, a $1$-scattered subspace will be simply called  \textbf{scattered}. 
\end{defin}

{From \cite[Theorem 2.3]{CsMPZ2019}, if $U$ is an $h$-scattered of $V(k,q^m)$, then $\dim_{\fq}(U)\leq \frac{km}{h+1}$. When the equality is reached, then $U$ is said be \textbf{maximum}.
}

Evasive subspaces are a special family of evasive sets, which were introduced first by Pudl\'ak and R\"odl \cite{pudlak2004pseudorandom}. These objects were then analyzed by Guruswami \cite{guruswami2011linear, guruswami2016explicit}, Dvir and Lovett \cite{dvir2012subspace} in connection with list decodability of codes with optimal rate and constant list-size.
A mathematical theory of evasive subspace {was recently developed } in \cite{BCMT}.

The following result highlights the relations between evasive subspaces and the parameters of the associated rank-metric codes.

\begin{theorem}[see \textnormal{\cite[Theorem 3.3]{CMNT}}]\label{thm:charact_evasive_genweights}
 Let $\mC$ be an $\Fmk$ code, and let $U\in\Phi([\mC])$. Then, the following are equivalent.
\begin{enumerate}[label=(\arabic*)]
 \item $U$ is an $(h,r)_q$-evasive subspace.
 \item $\dd_{\rk,k-h}(\mC) \geq n-r$.
 \item $\dd_{\rk,r-h+1}(\mC^\perp)\geq r+2$.
\end{enumerate}
  In particular, $\dd_{\rk,k-h}(\mC) = n-r$ if and only if $U$ is $(h,r)_q$-evasive but not $(h,r-1)_q$-evasive.
\end{theorem}

We conclude this section by recalling another family of $q$-systems which was recently introduced in \cite{ABNR}.

\begin{defin}
An $\Fmk$ system $U$ is said to be \textbf{$t$-cutting} if for every $\Fm$-subspace $H$ of $V(k,q^m)$ of codimension $t$ we have $\langle H \cap U \rangle_{\F_{q^m}} =H$. When $t=1$, we simply say that $U$ is \textbf{cutting} (or a \textbf{linear cutting blocking set}).
\end{defin}

The study of these objects was due to their connection to minimal rank-metric codes, as one can see from the following result.

\begin{theorem}[see \textnormal{\cite[Corollary 5.7]{ABNR}}]
Let $\mC$ be an $\Fmk$ code, and let $U \in \Phi([\mC])$ be any of the associated $\Fmk$ systems. Then, $\mC$ is a minimal rank-metric code if and only if $U$ is a linear cutting blocking set.
\end{theorem}

The following bound on the parameters of linear cutting blocking sets was derived in \cite{ABNR}.

\begin{prop}[see \textnormal{\cite[Corollary 5.10]{ABNR}}]\label{prop1}
Let $U$ be a cutting $\Fmk$ system, with $k \geq 2$. Then $n\geq m+k-1$.
\end{prop}

Moreover, always in \cite{ABNR} it was observed that linear cutting blocking sets are related with scattered subspaces when $k=3$. In this case, scattered subspaces were used to construct linear cutting blocking sets, as the following result shows.

\begin{prop}[see \textnormal{\cite[Theorem 6.3]{ABNR}}]\label{prop2}
If $U$ is a scattered $\Fms{n}{3}{m}$ system with $n \geq m+2$, then $U$ is cutting.
\end{prop}

In the next section we will further investigate the properties and the parameters of linear cutting blocking sets, and their {connection } with evasive subspaces. In particular, we will answer the following two natural questions:
Can we generalize the  Proposition \ref{prop2} to larger values of $k$? Does the converse of Proposition \ref{prop2} hold? The answers are both contained in a more general result given in Theorem \ref{Thm:caract}.

\section{Linear cutting blocking sets and evasive subspaces}\label{sec:cutting_evasive}

In this section we derive new results on evasive subspaces and linear cutting blocking sets. In particular, our aim is to prove the main result of Theorem \ref{Thm:caract} which shows how these two objects are related.

\subsection{Evasive subspaces}

We first start with evasive subspaces, showing in which cases an $(h,r)_q$-evasive $q$-system can also be  $(h',r')_q$-evasive, for some special parameters $(h',r')$.

The next result is an improvement on \cite[Proposition 2.6]{BCMT}.

\begin{prop}\label{prop:evasive}
Let $U$ be an $(h,r)_q$-evasive $\Fmk$ system and let $s<h$ be such that $n-r+s>\frac{(k-h+s)m}{2}$. Then $U$ is $(h-s,r-s-1)_q$-evasive.
\end{prop}
\begin{proof}
By \cite[Proposition 2.6]{BCMT}, the $\Fmk$ system $U$ is $(h-s+1,r-s+1)_q$-evasive and $(h-s,r-s)_q$-evasive.
Let $H$ be an $(h-s)$-dimensional $\Fqm$-subspace of $V(k,q^m)$ such that $\dim_{\fq} (H\cap U)=r-s$. Let us consider the projection $\pi_H:V(k,q^m)\rightarrow \Lambda$, where  $\Lambda \subseteq V(k,q^m)$ is a $(k-h+s)$-dimensional $\Fqm$-subspace  with $\Lambda \cap H=\{{0}\}$. In this way, $\pi_H(U)$ is an $\fq$-subspace of $\Lambda$ of $\fq$-dimension $n-r+s$. By the assumption on the parameters, $\pi_H(U)$ is not scattered, hence there exists a point $P=\langle v\rangle_{\Fm}$ with $\wt_{\pi_{H}(U)}(P)\ge 2$. This implies that the $(h-s+1)$-dimensional $\Fqm$-subspace spanned by $H$ and $\pi_H^{-1}(P)$ has weight at least $r-s+2$ in $U$, a contradiction. 
\end{proof}

Also, we have the following result.

\begin{prop}\label{prop:evasive1}
If $U$ is an $(h,r)_q$-evasive $\Fmk$ system with $r<hm$, then $U$ is $(2h,r+hm-1)_q$-evasive.
\end{prop}
\begin{proof}
By way of contradiction suppose that there exists a $(2h)$-dimensional $\Fqm$-subspace $H\subseteq V(k,q^m)$ with  $\wt_{U}(H)\geq r+hm$. Then each point of $H$ has weight at least $1$ in $U$ and $H$ can be partitioned in $q^{hm}+1$ $\Fqm$-subspaces of dimension $h$. Since $U$ is $(h,r)_q$-evasive, the maximum number of vectors of $U$ contained in $H$ is \[(q^r-1)(q^{hm}+1)+1,\] which is less than $q^{r+hm}$, a contradicion.
\end{proof}

\subsection{Linear cutting blocking sets}

We start by generalizing the result in Proposition \ref{prop2} to the case of arbitrary $k \geq 3$.
Furthermore, it is natural to ask whether the converse of Proposition \ref{prop2} is true, or at least under which condition it holds. Here we give a complete answer to this question.

\begin{theorem}\label{Thm:caract}
Let $U$ be an $\Fmk$ system. Then,  $U$ is $(k-2,n-m-1)_q$-evasive if and only if it is cutting.
\end{theorem}

\begin{proof}
 ($\Rightarrow$) Assume by contradiction that $U$ is not cutting. Then, there exists an  $\Fm$-hyperplane $H$  of $V(k,q^m)$ such that $\langle H\cap U\rangle_{\Fm}\subseteq M$, where $\dim_{\Fm}(M)=k-2$. In particular, we have $M\cap U\supseteq H\cap U$, and $\dim_{\fq}(H\cap U)\geq \dim_{\fq}(H)+\dim_{\fq}(U)-km=(k-1)m+n-km=n-m$. However, this contradicts the hypothesis that $U$ is $(k-2,n-m-1)_q$-evasive.
 
 ($\Leftarrow$) By way of contradiction, suppose that there exists a $(k-2)$-dimensional $\F_{q^m}$-subspace, say $M$, of $V(k,q^m)$ such that \[\dim_{\fq} (M\cap U)=h\geq n-m.\]
Since $U$ is cutting, each of the $q^m+1$ hyperplanes through $M$ has $q^t>q^h$ vectors in U. It follows that \[q^{n-m}(q-1)(q^m+1)+q^{n-m}\leq (q^t-q^h)(q^m+1)+q^h \leq |U|=q^n,\] a contradiction.
\end{proof}

From Theorem \ref{Thm:caract} we can actually characterize minimal rank-metric codes in terms of their second generalized rank weight. Even though it directly follows from Theorem \ref{Thm:caract} and Theorem \ref{thm:charact_evasive_genweights}, the following result will be named as a theorem, due to its importance and conciseness.

\begin{theorem}\label{thm:minimal_secondweight}
 Let $\mC$ be a nondegenerate $\Fmkdd$ code. Then, $\mC$ is minimal if and only if $d_2\geq m+1$.
\end{theorem}

Observe that  Theorem \ref{thm:minimal_secondweight} generalizes the result obtained in \cite[Corollary 6.19]{ABNR}, where the same equivalence was provided only for a special set of  parameters.

We know that there are some cases in which the lower bound in Proposition \ref{prop1} is reached. Indeed, there exists a scattered $\F_q$-subspace of $V(3,q^5)$ of dimension 7 \cite{BCMT} and hence a linear cutting blocking set of $V(3,q^5)$ of minimal dimension.
However, in general, the lower bound is not met for every set of parameters.
\begin{cor}\label{cor:nonexistence}
 If $m<(k-1)^2$ then there are no  cutting $[m+k-1,k]_{q^m/q}$ systems.
\end{cor}
\begin{proof}
 If $U$ is a cutting $[m+k-1,k]_{q^m/q}$ system, by Theorem \ref{Thm:caract}, $U$ is  $(k-2)$-scattered and by \cite[Theorem 2.3]{CsMPZ2019}, $\dim_{\fq}(U)\leq \frac{km}{k-1}$, which contradicts our assumptions.
\end{proof}

If we instead restrict to the case that $m$ is equal to $(k-1)^2$, we have that linear cutting blocking sets {of dimension $m+k-1$} are equivalent to {maximum} $(k-2)$-scattered $q$-systems. More precisely, we have the following.

\begin{cor}
 Let $U$ be a $[k(k-1),k]_{q^{(k-1)^2}/q}$ system and let $\cC$ be a code associated to $U$. The following are equivalent.
 \begin{enumerate}[label=(\arabic*)]
     \item $U$ is {maximum} $(k-2)$-scattered.
     \item $U$ is cutting.
     \item $\cC$ is MRD.
     \item $\cC$ is minimal.
 \end{enumerate}
\end{cor}

We conclude by providing two inductive constructions of linear cutting blocking sets.

\begin{prop}\label{prop:cuttingnoscatt}
Let $W$ be a linear cutting blocking set in a hyperplane $H$ of $V(k+1,q^m)$ of dimension $t$ and let $\langle v \rangle_{\Fqm}$ be a point not in $H$. Then $U=W+\langle v \rangle_{\Fqm}$ is a cutting $[t+m,k+1]_{q^m/q}$ system.
\end{prop}
\begin{proof}
Easy computations show that $U$ is a union of lines through $\langle v \rangle_{\Fqm}$. Let $\Pi$ be a hyperplane of $V(k+1,q^m)$ different from $H$. Since $W$ is cutting, $\langle\Pi\cap H \cap W\rangle_{\fq}=H\cap \Pi$ and there exists at least a line trough $\langle v \rangle_{\Fqm}$ intersecting $\Pi$ in a point defined by a vector of $U\setminus W$. The assertion follows.
\end{proof}

\begin{prop}
Let $U_i$, $i\in\{1,2\}$, be a  cutting $[n_i,k_i]_{q^m/q}$ system of $V_i=V(k_i,q^m)$. If $U_i$ is $(k_i-1,n_i-m-1+t_i)_q$-evasive for each $i\in\{1,2\}$, with $t_1,t_2\geq 0$ and $t_1+t_2<m+2$, then $U_1\oplus U_2$ is a cutting $[n_1+n_2,k_1+k_2]_{q^m/q}$ system.
\end{prop}
\begin{proof}
Let $\Pi$ be a hyperplane of $V_1\oplus V_2$. Then $\dim_{\fq}(\Pi)=(k_1+k_2-1)m$. Also, $\dim_{\fq}(\Pi\cap U_i)\leq n_i+m-1$ and
\begin{align*}
    \dim_{\fq}(\Pi\cap (U_1\oplus U_2))&\geq \dim_{\fq}(\Pi)+\dim_{\fq}(U_1\oplus U_2)-\dim_{\fq}(V_1\oplus V_2)\\
    &= (k_1+k_2-1)m+n_1+n_2-(k_1+k_2)m=n_1+n_2-m.
\end{align*}
Since $(\Pi\cap V_1)\oplus (\Pi\cap V_2)$ is a hyperplane of $H$ and $t_1+t_2<m+2$, we get 
\[\dim_{\fq}\left((\Pi\cap V_1)\oplus (\Pi\cap V_2)\right)\leq n_1-m-1-t_1+n_2-m-1-t_2<n_1+n_2-m,\] and hence the assertion.
\end{proof}

	\subsection{Dual of a linear cutting blocking set}
	
	Let $\sigma: V\times V\longrightarrow \Fm$ be a nondegenerate
	bilinear form on $V=V(k,q^m)$ and define
	\[ \begin{array}{rccl}
		\sigma' \colon & V\times V &\longrightarrow & \F_q,\\
       &({u}, { v}) 	& \longmapsto &	\Tr_{q^m/q}(\sigma({u}, {v})),
	\end{array} \]
	where $\Tr_{q^m/q}$ denotes the trace function of $\F_{q^m}$ over $\F_q$.
	Then $\sigma'$ is a nondegenerate bilinear form on
	$V$, when $V$ is  regarded as a $km$-dimensional vector space  over
	$\fq$. Let $\tau$ and $\tau'$ be the orthogonal complement maps
	defined by $\sigma$ and $\sigma'$ on the lattices  of the
	$\Fm$-subspaces and $\fq$-subspaces of $V$, respectively.
	Recall that  if $R$ is an
	$\Fm$-subspace of $V$ and $U$ is an $\fq$-subspace of $V$
	then $U^{\tau'}$ is an $\fq$-subspace of $V$, $\dim_{\Fm}(R^\tau)+\dim_{\Fm}(R)=k$, and
	$\dim_{\fq}(U^{\tau'})+\dim_{\fq}(U)= km$. It easy to see
	that $R^\tau=R^{\tau'}$ for each $\Fm$-subspace $R$ of $V$. For a more detailed explanation, we refer to \cite[Chapter 7]{Taylor}).
	
	With the notation above, $U^{\tau '}$ is called the \textbf{dual} of $U$ (with respect to $\tau'$). 
	Up to $\mathrm{GL}(k,q^m)$-equivalence, the dual of an $\fq$-subspace of $V$ does not depend on the choice of the nondegenerate bilinear forms $\sigma$ and $\sigma'$ on $V$. For more details see \cite{Polverino}. If $R$ is an $s$-dimensional $\Fm$-subspace of $V$ and $U$ is a $t$-dimensional $\fq$-subspace of $V$, then
	\begin{equation}\label{pesi}
		\dim_{\fq}(U^{\tau'}\cap R^\tau)-\dim_{\fq}(U\cap R)=km-t-sm.
	\end{equation}
	From \eqref{pesi} and from Theorem \ref{Thm:caract} the next results immediately follow.
	
	\begin{prop}
	Let $U$ be a  scattered $[\frac{km}{2},k]_{q^m/q}$ system. Then the following are equivalent:
	\begin{enumerate}[label=(\arabic*)]
	    \item $U$ is $(2,m-1)_q$-evasive,
	    \item  $U^{\tau'}$ is a $\left(k-2,\frac{km}{2}-m-1\right)_q$-evasive $[\frac{km}{2},k]_{q^m/q}$ system,
	    \item  $U^{\tau'}$ is a cutting $[\frac{km}{2},k]_{q^m/q}$ system. 
	\end{enumerate}
	\end{prop}
	\begin{cor}
		Let $U$ be a  scattered $[2m,4]_{q^m/q}$ system. Then the following are equivalent:
	\begin{enumerate}[label=(\arabic*)]
	\item $U$ is cutting,
	    \item $U$ is $(2,m-1)_q$-evasive,
	    \item  $U^{\tau'}$ is a  cutting $[2m,4]_{q^m/q}$ system,
	    \item  $U^{\tau'}$ is a $(2,m-1)_q$-evasive  $[2m,4]_{q^m/q}$ system.
	\end{enumerate}
	\end{cor}

\subsection{The case $(k,m)=(4,3)$}

We conclude this  section with a focus on a special case, namely  $(k,m)=(4,3)$. In particular, in this  case we will determine the shortest length that a minimal rank-metric code of dimension $k$ can have with respect to the field extension $\Fqm/\fq$. 

Formally, denote by $c_q(k,m)$  the smallest dimension of a linear cutting blocking set in $V(k,q^m)$, or, equivalently, the length of the shortest minimal rank-metric code of dimension $k$, that is
$$c_q(k,m):= \min\{\  n \in \N \,: \mbox{ there exists a cutting } \Fmk \mbox{ system } \}.$$
 From Proposition \ref{prop:cuttingnoscatt} we can  deduce the recursive inequality given by
$$ c_q(k,m)\leq c_q(k-1,m)+m. $$
On the other hand, Proposition \ref{prop1} can be rewritten as
$$c_q(k,m)\geq k+m-1,$$
and we have seen that this is not always an equality; see Corollary \ref{cor:nonexistence}.

In this section we determine the exact value  $c_q(4,3)$  for every prime power $q$.
By Corollary \ref{cor:nonexistence}, we can already deduce that $c_q(4,3)\geq 7$. We will see that actually $c_q(4,3)>7$ and this will be a consequence of the following more general bound on evasive $q$-systems.

\begin{lemma} Let $h \in \mathbb N$ be such that $h\leq k$. Let $U$ be an $(h,h+1)_q$-evasive $\Fmk$ system. Then, for every $1\leq t \leq h-1$, we have
$$n\leq \max\left\{ \frac{km}{h-t+1}, \frac{(k-h+t)m+(h-t+1)(t+1)}{t+1}
\right\}.$$
\end{lemma}
\begin{proof}
By way of contradiction, suppose that  $U$ is an $(h,h+1)_q$-evasive $\Fmk$ system with 
$$n > \max\left\{ \frac{km}{h-t+1}, \frac{(k-h+t)m+(h-t+1)(t+1)}{t+1}
\right\}.$$ Let us fix $1\leq t \leq h-1$.
Since $n>\frac{km}{h-t+1}$, then $U$ cannot be $(h-t)$-scattered, and hence it must be $(h-t,h-t+1)_q$-evasive. Thus, there exists an $\Fm$-subspace $H\subseteq V(k,q^m)$ with $\dim_{\Fm}(H)=h-t$ and of weight $h-t+1$ in $U$.  Let us project $U$ from $H$ over a $(k-h+t)$-dimensional $\Fqm$-subspace $\Lambda$, with $\Lambda \cap H=\{ 0\}$, and let us denote this projection map by $\pi_H:V(k,q^m)\rightarrow \Lambda$. Since $\ker(\pi_H)\cap U=H\cap U$, by restricting the map $\pi_H$ to $U$ we obtain an $\F_q$-subspace $W:=\pi_H(U)\subseteq \Lambda$ that is an $[n-h+t-1,k-h+t]_{q^m/q}$ system. By assumptions on $n$, the $q$-system $W$ is not $t$-scattered and hence there is a $t$-dimensional subspace $H'$ of $\Lambda$ of weight at least $t+1$ in $W$. Therefore, the space $\pi_H^{-1}(H')$ is an $h$-dimensional $\Fm$-subspace and $\pi_H^{-1}(H')\cap U=\pi_H^{-1}(W)$, which has dimension $\dim_{\fq}(W\cap H')+\dim_{\fq}(H\cap U)\geq (h-t+1)+(t+1)=h+2$, a contradiction.
\end{proof}
\begin{cor}\label{cor:nocutting}
There are no  cutting $[7,4]_{q^3/q}$ systems.
\end{cor}
\begin{proof}
 By Theorem \ref{Thm:caract}, if $U$ is a  cutting $[7,4]_{q^3/q}$ system then $U$ is $(2,3)_q$-evasive. The result follows from the previous lemma with $k=4$, $m=3$, $h=2$, and $n=7$.
\end{proof}

\begin{prop}\label{prop:cutting8}
Let $u\in\Fq3\setminus\F_q$, the $\F_q$-subspace
\[U=\left\{(\alpha_0+\alpha_1u,\alpha_2+\alpha_3u,\alpha_4+\alpha_5u,\alpha_6+\alpha_7u)\,:\ \alpha_i\in\F_q\right\}\] is a cutting $[8,4]_{q^3/q}$ system.
\end{prop}
\begin{proof}
By Theorem \ref{Thm:caract}, this is equivalent to show that $U$ is $(2,4)_q$-evasive in $V(4,q^3)$ and by Propositions \ref{prop:evasive} and \ref{prop:evasive1} this is equivalent to show that $U$ is $(1,2)_q$-evasive.
Hence, it is sufficient to show that there are no points of weight $3$ in $U$. By way of contradiction, suppose that a point $P=\langle(\alpha_0+\alpha_1u,\alpha_2+\alpha_3u,\alpha_4+\alpha_5u,\alpha_6+\alpha_7u)\rangle_{\Fqm}$ has weight $3$ in $U$. Then, for each $\lambda\in\Fq3$ there exist $\alpha'_i\in\F_q$, $i\in\{0,\ldots,7\}$, such that
\begin{equation}\label{eq:1}
\lambda(\alpha_{2i}+\alpha_{2i+1}u)=\alpha'_{2i}+\alpha'_{2i+1}u,
\end{equation}
for each $i\in\{0,\ldots,3\}$.
Taking $\lambda=u$ and $\lambda=u^2$, from Equation \eqref{eq:1}, we get $\alpha_i=0$, for each $i\in\{0,\ldots,7\}$, a contradiction.
\end{proof}

In other words, Corollary \ref{cor:nocutting} and Proposition \ref{prop:cutting8} mean that for every prime power $q$, we have
$$ c_q(4,3)=8.$$

\section{A $(2,3)_q$-evasive $[8,4]_{q^4/q}$ system and the associated rank-metric code}\label{sec:construction_evasive}

This section is dedicated to the existence of a linear cutting blocking set of dimension $8$ in $V(4,q^4)$. Or, in other words, to the existence of a scattered $[8,4]_{q^4/q}$ system which is also $(2,3)_q$-evasive. 

The importance of this result is multiple, and it is related to the theory of rank-metric codes. On the one hand it  produces a construction of a minimal rank-metric code of dimension $4$ with respect to the field extension $\F_{q^4}/\fq$ with shortest length.
On the other hand it provides the first example of a $[2m,2(m-d+1),d]_{q^m/q}$ MRD code which is not the direct sum of two $[m,m-d+1,d]_{q^m/q}$ MRD codes. 

For the rest of this section we define the $[8,4]_{q^4/q}$ system $U$ to be the subspace of $\F_{q^4}^4$
\begin{equation}\label{formU}
U := \left\{\left(x,y,x^q+y^{q^2},x^{q^2}+y^q+y^{q^2}\right) \, : \, x,y \in \mathbb{F}_{q^4} \right\}.
\end{equation}

\subsection{Scatteredness and evasiveness}

In this section we only assume to be working over a finite field of characteristic $2$.  As a first step, we show for which finite fields $\fq$ the $[8,4]_{q^4/q}$ system $U$ is scattered.

\begin{theorem}\label{thm:Uscattered}
Let $q=2^{h}$. Then $U$ is maximum scattered if and only if $h\not \equiv 2 \pmod{4}$.
\end{theorem}
\begin{proof}
The $\fq$-dimension of $U$ is clearly $8$. Let $\lambda \notin \mathbb{F}_{q}$ be such that 
$$(x,y,x^q+y^{q^2},x^{q^2}+y^q+y^{q^2}) = \lambda (u,v,u^q+v^{q^2},u^{q^2}+v^q+v^{q^2}).$$
The set $U$ is maximum scattered if and only if the previous equation holds only for $u=v=0$.

The above equation yields 
$x=\lambda u$, $y=\lambda v$, 
$$\lambda^q u^q+\lambda^{q^2}v^{q^2}=\lambda(u^q+v^{q^2}), \quad \lambda^{q^2}u^{q^2}+\lambda^qv^{q}+\lambda^{q^2} v^{q^2}=\lambda(u^{q^2}+v^q+v^{q^2}).$$
This implies 
$$(\lambda^{q^2}+\lambda)v^{q^2} +(\lambda^q+\lambda) u^q=0, \quad (\lambda^{q^2}+\lambda)(v^{q^2}+u^{q^2}) +(\lambda^q+\lambda) v^q=0,$$
that we rewrite as
$$ F(u,v):=v^{q^2}\lambda^{q^2} + u^q\lambda^q+(v^{q^2}+u^q)\lambda =0, \quad G(u,v):=(v^{q^2}+u^{q^2})\lambda^{q^2}+v^q\lambda^q+(u^{q^2}+v^q+v^{q^2})\lambda=0.$$
If $v=0$, then $F(u,v)=0$ implies $u=0$.

So we suppose $v\neq 0$. Then, $(v^{q^2}+u^{q^2})F(u,v)+v^{q^2}G(u,v)=0$ reads

\begin{equation}\label{Eq:lambda}
\Big((v^{q^2}+u^{q^2})u^q+v^{q^2}v^q\Big)(\lambda^q+\lambda)=0.
\end{equation}
Since $\lambda\notin \F_q$, then
$$v^{q}u+u^{q+1}+v^{q+1}=0.$$
Since the above expression is homogeneous, we may assume $v=1$. 
 Consider the polynomial $\varphi(T):=T^{q+1}+T+1$. Any root $t$ of $\varphi(T)$ satisfies $t^{q+1}+t+1=0$ and thus  $t^{q}=1+1/t$, $t^{q^2}=1+1/t^q=1/(t+1)$, $t^{q^3}=t$. This shows that $t\in \mathbb{F}_{q^3}\cap \mathbb{F}_{q^4}=\mathbb{F}_q$ and so $t^2+t+1=0$. From this we further deduce that 
 the only possible roots of $\varphi(T)$ in $\F_{q^4}$ are in $\{\omega,\omega^2\}\cap \fq$, where $\F_4=\{0,1,\omega,\omega^2\}$.
 
\begin{enumerate}

\item Let $h\equiv 1 \pmod{2}$. In this case we have $\{\omega,\omega^2\}\cap \mathbb{F}_{q}=\emptyset$. Thus, there are no nonzero pairs $(u,v)\in \mathbb{F}_{q^4}^2$ such that $v^{q}u+u^{q+1}+v^{q+1}=0$. So $U$ is maximum scattered.

\item Let $h\equiv 0 \pmod{4}$. In this case $\{\omega,\omega^2\}\subseteq \mathbb{F}_q$. Thus $v=\omega^i u$, for some $i=1,2$ and for a given $u\in\F_{q^4}$. We want to show that the equation $$\phi_u(\lambda):=F(u,\omega^i u)=\omega^iu^{q^2}\lambda^{q^2} + u^q\lambda^q+(\omega^iu^{q^2}+u^q)\lambda=0$$ admits at most $q$ solutions in $\lambda$, i.e. $\phi_u(\lambda)=0$ if and only if $\lambda\in\F_q$.

Straightforward computations show that if $t\in\F_{q^4}$ is a solution of $\phi_u(\lambda)=0$ then $t$ satisfies 
$$(1+\omega^i \gamma^{q^2}+\omega^i \gamma^q)t^q+(\omega^i\gamma^{q^2+1}+\gamma+1)t=0.$$
Hence, there are more than $q$ solutions only if $1+\omega^i \gamma^{q^2}+\omega^i \gamma^q=\omega^i \gamma^{q^2+1}+\gamma+1=0.$ From the second equation we get $\omega^i \gamma^{q^2+1}+\gamma^{q^2}+1=0$ and thus $\gamma\in \mathbb{F}_{q^2}$. So $\omega \gamma^{2}+\gamma+1=0$. Since $\Tr_{q/2}(\omega)=0$ (recall that $q=2^{4h^{\prime}}$), $\gamma$ belongs to $\mathbb{F}_q$ and this contradicts $0=1+\omega^i \gamma^{q^2}+\omega^i \gamma^q=1$. This shows that $F(u,\omega^i u)=0$ if and only if $\lambda\in\F_q$. Thus, $U$ is  scattered.

\item Let $h\equiv 2 \pmod{4}$. Also in this case $\{\omega,\omega^2\}\subseteq \mathbb{F}_q$ and $u=\omega^i v$, for some  $i=1,2$.  Consider $\overline{v}$ such that $\overline{v}^2+\overline{v}+\omega^2=0$ and $\overline{u}=\omega \overline{v}$. Since $h\equiv 2 \pmod{4}$, $\Tr_{q/2}(\omega^2)=1$, $\overline{v}\in \mathbb{F}_{q^2}\setminus \mathbb{F}_q$, and $\overline{u}^q=\overline{v}+1$. In this case \eqref{Eq:lambda} vanishes and $F(u,v)$ and $G(u,v)$ are proportional. To prove that $U$ is not scattered it is enough to prove that $F(\omega \overline{v},\overline{v})= \overline{v} \lambda^{q^2}+\omega (\overline{v}+1)\lambda^q +(\overline{v}+\omega+\omega \overline{v})\lambda$ has $q^2$ roots in $\mathbb{F}_{q^4}$. To see this, note that any solution $\lambda \in \mathbb{F}_{q^4}$ satisfies 
\begin{eqnarray*}
\lambda^{q^2}&=&\omega \gamma \lambda^q+(1+\omega \gamma)\lambda,
\end{eqnarray*}
where $\gamma = \frac{\overline{v}+1}{\overline{v}}$ satisfies $\gamma^{q}=1/\gamma=\gamma +\omega$. Thus 
\begin{eqnarray*}
\lambda^{q^3}&=&\omega \gamma^q \lambda^{q^2}+(1+\omega \gamma^q)\lambda^q\\
&=&\omega \gamma^q \left(\omega \gamma \lambda^q+(1+\omega \gamma)\lambda \right)+(1+\omega \gamma^q)\lambda^q\\
&=&\omega^2 \lambda^q+ \omega \gamma^q (1+\omega \gamma)\lambda +(1+\omega \gamma^q)\lambda^q\\
&=&(\omega^2+1+\omega \gamma^q) \lambda^q+ \omega \gamma^q (1+\omega \gamma)\lambda,
\end{eqnarray*}
and 
\begin{eqnarray*}
\lambda^{q^4}&=&(\omega^2+1+\omega \gamma) \lambda^{q^2}+ \omega \gamma (1+\omega \gamma^q)\lambda^q\\
&=&(\omega^2+1+\omega \gamma) (\omega \gamma \lambda^q+(1+\omega \gamma)\lambda)+ \omega \gamma (1+\omega \gamma^q)\lambda^q=\lambda.
\end{eqnarray*}
Therefore, $F(\omega \overline{v},\overline{v})$ has  $q^2$ roots in $\mathbb{F}_{q^4}$ and $U$ is not scattered. 
\end{enumerate}
\end{proof}

We now switch to show that the $[8,4]_{q^4/q}$ system $U$ is also $(2,3)_q$-evasive. In order to do this, we first need three technical auxiliary results.

\begin{prop}\label{Prop:caso1}
Let $q=2^{h}$, $h\not\equiv 2\pmod {4}$,  $A,B,C,D\in \mathbb{F}_{q^4}$, $B^{q^2+1}\neq 1$, and 
$$\Lambda =\{(u,v) : v^{q^2}+Au+Bv+u^q=v^{q^2}+u^{q^2}+v^q+Cu+Dv=0\}\cap \mathbb{F}_{q^4}^2.$$
Then $\#\Lambda\leq q^3$.
\end{prop}
\begin{proof}
Denote by $F(u,v)$ and $G(u,v)$ the {quantities } $v^{q^2}+Au+Bv+u^q$ and $v^{q^2}+u^{q^2}+v^q+Cu+Dv$ respectively{, where $u,v\in \mathbb{F}_{q^4}$}. Now, we have
$$B^{q^2}F(u,v)+(F(u,v))^{q^2}=(B^{q^2+1}+1)v+ AB^{q^2} u + B^{q^2}u^q+ A^{q^2}u^{q^2} + u^{q^3}$$
and therefore, if $(u,v)\in\Lambda$, then
\begin{equation}\label{Eq:v}
    v=\frac{AB^{q^2} u + B^{q^2}u^q+ A^{q^2}u^{q^2} + u^{q^3}}{B^{q^2+1}+1}.
\end{equation}
Thus, we can compute
$$(B^{q^2+1}+1)^{q+1}G\left(u,\frac{AB^{q^2} u + B^{q^2}u^q+ A^{q^2}u^{q^2} + u^{q^3}}{B^{q^2+1}+1}\right)=\alpha u +\beta u^q+\gamma u^{q^2}+\delta u^{q^3},$$
where
\begin{eqnarray*}
 \alpha &=&A B^{q^3+q^2+q}D + AB^{q^3+q}+ AB^{q^2}D + A + B^{q^3+q^2+q+1}C\\
    &&+ B^{q^2+1}C +B^{q^2+1} + B^{q^3+q}C + C + 1,\\
    \beta&=&A^q B^{q^3+q^2+1} + A^{q}B^{q^3}+B^{q^3+q^2+q}D+B^{q^3+q}+B^{q^2}D+1,\\
    \gamma &=&A^{q^2}B^{q^3+q+1} +A^{q^2}B + A^{q^2}B^{q^3+q}D + A^{q^2}D\\
    &&+ B^{q^3+q^2+q+1} + B^{q^3+q^2+1} + B^{q^2+1} + B^{q^3+q}+ B^{q^3} + 1,\\
    \delta &=&A^{q^3}B^{q^2+1} + A^{q^3} + B^{q^3+q+1} + B + B^{q^3+q}D+D.
\end{eqnarray*}

If at least one among $\alpha,\beta,\gamma,\delta$ is not zero, then $\alpha u +\beta u^q+\gamma u^{q^2}+\delta u^{q^3}=0$ provides at most $q^3$ possibilities for $u$, which together with \eqref{Eq:v} yields  $\#\Lambda\leq q^3$.

Suppose now that  $\alpha=\beta=\gamma=\delta =0$. From $\delta=0$ we get
$$D=\frac{A^{q^3}B^{q^2+1} + A^{q^3} + B^{q^3+q+1} + B}{B^{q^3+q}+1}$$
and combining it with $\beta=\gamma=0$ we obtain 
\begin{eqnarray*}
(B^{q^2+1} + 1)(A^qB^{q^3} + A^{q^3}B^{q^2} + B^{q^3+q} + 1)&=&0,\\
(B^{q^2+1} + 1)(A^{q^3+q^2} + B^{q^3+q}+B^{q^3} + 1)&=&0.\\
\end{eqnarray*}
Note that $B\neq0$. Furthermore, we must have $A\neq0$, otherwise $B^{q^3+q} + 1=(B^{q^2+1} + 1)^q=0$ and so $B^{q^3+q}+B^{q^3} + 1\neq 0$ too. 

Also,  $B\in \mathbb{F}_q$ implies 
$$A^{q^3+q^2}=B^2+B+1\in \mathbb{F}_q$$
and so $A\in \mathbb{F}_{q^2}\setminus\{0\}$.
The equation $A^qB^{q^3} + A^{q^3}B^{q^2} + B^{q^3+q} + 1=0$ reads $B^2+1=0$, that is $B=1$, a contradiction to the assumption $B^{q^2+1}\neq 1$.

From now on $B\notin \mathbb{F}_q$. First we want to determine the solutions in $A$ to $A^qB^{q^3} + A^{q^3}B^{q^2} + B^{q^3+q} + 1=0$. Note that if there exist pairs $(x,y)\in \mathbb{F}_{q^4}^2$ such that  
\begin{eqnarray*}
B^{q^3}x + B^{q^2}y + B^{q^3+q} + 1&=&0\\
 Bx +B^{q}y + B^{q^3+q} + 1&=&0\\
\end{eqnarray*}
then either $B^{q^3+q}\neq B^{q^2+1}$ or $B^{q^3+q}= B^{q^2+1}$ and $(B^{q^3+q}+1)(B^{q^3}+B)=0$. The latter condition yields $B\in \mathbb{F}_q$, a contradiction, so we only need to consider $B^{q^3+q}\neq B^{q^2+1}$ and 
$$x=\frac{(B^{q^3+q}+1)(B^{q^2}+B^q)}{B^{q^3+q}+ B^{q^2+1}}.$$
This means that $A^{q}=\frac{(B^{q^3+q}+1)(B^{q^2}+B^q)}{B^{q^3+q}+ B^{q^2+1}}$.

Now, $A^{q^3+q^2} + B^{q^3+q}+B^{q^3} + 1=0$ yields

\begin{eqnarray}\label{Eq:finale}
    H(B):=B^{q^3+q^2}(B^{q^3+q}+ B^{q^2} + 1)B^2+(B^{q^3+q} + 1)(B^{q^3+q^2} + 1)(B^q + B)^{q^2}B\nonumber\\
    +B^{3q^3+3q} + B^{3q^3+2q} + B^{2q^3+2q} + B^{2q^3+q^2+q}+ B^{3q^3+q}+ B^{q^3+q^2} + B^{2q^3}&=&0.
\end{eqnarray}
By direct checking, $B^{q^3}H(B)+B^q(H(B))^{q^2}=(B^{q^2} + B)^{q}(B^{q^2+1} + 1)^{2q}(B^{q^3+q}+ B^{q^2+1})$, whose vanishing gives $B\in \mathbb{F}_{q^2}$. 

Thus we need only to consider the case $B \in \mathbb{F}_{q^2}\setminus \mathbb{F}_q$.
In this case $H(B)=0$ reads $(B^q + B)^2(B^{q+2} + B^{4q}+ B^{3q} + 1)=0$. Since $B\notin \mathbb{F}_q$, $B^2=B^{3q}+B^{2q}+1/B^q$ and $B^{2q}=B^{3}+B^{2}+1/B\neq 0$. So 
$$B^4=B^{6q}+B^{4q}+1/B^{2q}= (B^{3}+B^{2}+1/B)^3+(B^{3}+B^{2}+1/B)^2+\frac{1}{B^{3}+B^{2}+1/B},$$
that is 
$$B^{16} + B^{13} + B^{9} + B^{5} + B + 1=(B+1)^{12}(B^2+B+\omega)(B^2+B+\omega^2)=0,$$
where $\mathbb{F}_4^*=\langle\omega\rangle$. 

\begin{enumerate}
    \item If $h\equiv 1 \pmod 2$ then $\omega \in \mathbb{F}_{q^2}\setminus \mathbb{F}_q$ and from $\Tr_{q^2/2}(\omega)=\Tr_{q^2/2}(\omega^2)=1$ and any $B$ satisfying $(B^2+B+\omega)(B^2+B+\omega^2)=0$ does not belong to $\mathbb{F}_{q^2}$, a contradiction. 
    \item If $h\equiv 0 \pmod 4$ then $\omega \in \mathbb{F}_{q}$ and from $\Tr_{q/2}(\omega)=\Tr_{q/2}(\omega^2)=0$ we conclude that  any $B$ satisfying $(B^2+B+\omega)(B^2+B+\omega^2)=0$ belongs to $\mathbb{F}_{q}$, again a contradiction.
\end{enumerate}
This provides the final contradiction to $\alpha=\beta=\gamma=\delta=0$ and therefore $\Lambda\leq q^3$.
\end{proof}

\begin{prop}\label{Prop:caso2}
Let $q=2^{h}$,  $A,B,C,D\in \mathbb{F}_{q^4}$, $B^{q^2} =1/B$, and 
$$\Lambda =\{(u,v) : v^{q^2}+Au+Bv+u^q=v^{q^2}+u^{q^2}+v^q+Cu+Dv=0\}\cap \mathbb{F}_{q^4}^2.$$
Then $\#\Lambda\leq q^3$.
\end{prop}
\begin{proof}
We use the same notation as in the proof of Proposition \ref{Prop:caso1}. Now,
\begin{equation}\label{Eq:1}
B^{q^2}F(u,v)+(F(u,v))^{q^2}=AB^{q^2} u + B^{q^2}u^q+ A^{q^2}u^{q^2} + u^{q^3}=0
\end{equation}
and 
\begin{equation}\label{Eq:2}
F(u,v)+G(u,v)=(A+C)u + u^q + u^{q^2} + (B+D)v + v^q=0.
\end{equation}
If $A=0$, then $AB^{q^2} u + B^{q^2}u^q+ A^{q^2}u^{q^2} + u^{q^3}=0$ has at most $q^2$ solutions and this provides at most $q^3$ pairs, since for each $u$ there are at most $q$ values $v$ such that $(A+C)u + u^q + u^{q^2} + (B+D)v + v^q=0$. So we can suppose $A\neq 0$. With the same argument as above, we consider only the case in which $AB^{q^2} u + B^{q^2}u^q+ A^{q^2}u^{q^2} + u^{q^3}$ has $q^3$ roots in $\mathbb{F}_{q^4}$ which is equivalent to $A^{q+1}=B^q$. Note that
$$
(F(u,v)+G(u,v))^q+(B+D)^q(F(u,v)+G(u,v))+F(u,v)=\epsilon v+\alpha u+\beta u^q+\gamma u^{q^2}+\delta u^{q^3},
$$
where
\begin{eqnarray*}
\alpha&=&A(B+D+1)^q+ C(B+D)^q\\
\beta&=&(A+B+C+D+1)^q\\
\gamma&=&B^q+D^q+1\\
\delta&=&1.
\end{eqnarray*}
If $\epsilon\neq 0$ then $\#\Lambda\leq q^3$. Assume $\epsilon= 0$. From
\begin{eqnarray}\label{Eq:grado2}
AB^{q^2} u + B^{q^2}u^q+ A^{q^2}u^{q^2} + u^{q^3}+\alpha u+\beta u^q+\gamma u^{q^2}+\delta u^{q^3}\nonumber\\
=(A^{q^2}+B^q+D^q+1)u^{q^2}+(A+B+C+D+1+B^q)^qu^q\nonumber\\
+(AB^{q^2}+A(B+D+1)^q+ C(B+D)^q)u=0,
\end{eqnarray}
we conclude that if one among 
$$A^{q^2}+B^q+D^q+1, A+B+C+D+1+B^q, A(B^{q}+B+D+1)^q+ C(B+D)^q$$

does not vanish, then at most $q^2$ values of $u$ satisfy \eqref{Eq:grado2} and  $\#\Lambda\leq q^3$. 
So 
$$D=A^{q}+B+1, \quad C=A+B+D+1+B^q=A+A^{q}+B^q.$$
From $A(B^{q}+B+D+1)^q+ C(B+D)^q=0$, we get 
$$A(A^{q+1}+A^q)^q+ (A+A^q+A^{q+1})(A^q+1)^q=0.$$
Thus

$$A^{q^2+q}+A^{q+1}+A^q+A=0.$$
Note that if $A^{q+1}+A^q+A=0$ then $A=0$, a contradiction. So we can suppose $A^{q+1}+A^q+A\neq 0$. Thus 

$$A^{q^2}=1+A+A/A^q, \quad A^{q^3}=1+A^q+\frac{A^q}{1+A+A/A^q}, \quad A=1+1+A+A/A^q+\frac{1+A+A/A^q}{1+A^q+\frac{A^q}{1+A+A/A^q}}$$
and finally 
$$A^{q+1}+A^2+A^{2q+2}=A^{2q}+A^{2q+2}+A^2.$$

Then $A\in \mathbb{F}_q^*$, $B=A^2\in \mathbb{F}_q$. Thus, from  $B^{q^2+1}=1$, $B=1=A=C=D$. In this case Equations \eqref{Eq:1} and \eqref{Eq:2} yield  $$u^{q^3}+u^{q^2}+u^{q}+u=0=v+u^{q}+u$$
and  $\#\Lambda\leq q^3$. 
\end{proof}

\begin{prop}\label{Prop:caso3}
Let $q=2^{h}$, $h\equiv 1\pmod 2$, and let $A_{1,3},A_{1,4},A_{2,3},A_{2,4},A_{3,4}\in \mathbb{F}_{q^4}$ be such that $(A_{1,3},A_{2,3},A_{3,4})\neq (0,0,0)$. Then the following system has at most $q^2$ solutions $(u,v)\in\F_{q^4}^2$:
$$\begin{cases}
A_{2,3}u+A_{1,3} v=0\\
A_{2,4}u+A_{1,4} v=0\\
A_{3,4}v+A_{2,4}(u^q+v^{q^2})+A_{2,3}(u^{q^2}+v^{q^2}+v^q)=0\\
A_{3,4}u+A_{1,4}(u^q+v^{q^2})+A_{1,3}(u^{q^2}+v^{q^2}+v^q)=0.
\end{cases}$$
\end{prop}
\begin{proof}
If $A_{2,3}=0$ and $A_{1,3}\neq 0$, then $v=0$ from the first equation and from the fourth one there are at most $q^2$ distinct pairs $(u,0)$.

If $A_{2,3}=0=A_{1,3}$, then $A_{3,4}\neq 0$. If $(A_{2,4},A_{1,4})=(0,0)$, then the unique solution is $(0,0)$. If $A_{2,4}\neq0$, then $A_{2,4}u+A_{1,4} v=0=A_{3,4}v+A_{2,4}(u^q+v^{q^2})$ provide at most $q^2$ solutions. The same holds if $A_{1,4}\neq0$. 

Suppose now $A_{2,3}\neq 0$. From $A_{2,3}u+A_{1,3} v=0$ one gets $u=\lambda v$ and 

$$\begin{cases}
(A_{2,4}\lambda +A_{1,4}) v=0\\
A_{3,4}v+A_{2,4}(\lambda^qv^q+v^{q^2})+A_{2,3}(\lambda^{q^2} v^{q^2}+v^{q^2}+v^q)=0\\
A_{3,4}\lambda v+A_{1,4}(\lambda^q v^q+v^{q^2})+A_{1,3}(\lambda^{q^2} v^{q^2}+v^{q^2}+v^q)=0.
\end{cases}$$
The three equations above vanish  if and only if
$$\begin{cases}
A_{1,4}=A_{2,4}\lambda\\
A_{3,4}=0\\
A_{2,3}=A_{2,4}\lambda^q\\
A_{2,4}+A_{2,3}(\lambda^{q^2}+1)=0\\
A_{1,3}=A_{1,4}\lambda^q\\
A_{1,4}+A_{1,3}(\lambda^{q^2}+1)=0.\\
\end{cases}$$
Since $A_{2,4}\neq 0$, 
$$\lambda^{q^2+q}+\lambda^q+1=0,$$
that is $(\lambda^q)^{q+1}+(\lambda^q)+1=0$ which has no solution if $h\equiv 1 \pmod{2}$. 
\end{proof}

We are now ready to prove that when $q$ is an odd power of $2$, the $[8,4]_{q^4/q}$ system $U$ is $(2,3)_q$-evasive.
\begin{theorem}\label{thm:Uevasive}
Let $q=2^{h}$, $h\equiv 1 \pmod{2}$. Then $U$ is $(2,3)_q$-evasive.
\end{theorem}
\begin{proof}
We need to prove that any $2$-dimensional $\F_{q^4}$-subspace contains at most $q^3$ vectors of $U$.

Let $P=(x,y,x^q+y^{q^2},x^{q^2}+y^q+y^{q^2})$ and $Q=(z,t,z^q+t^{q^2},z^{q^2}+t^q+t^{q^2})$ be two  $\F_{q^4}$-independent  vectors of $U$. A vector $R=(u,v,u^q+v^{q^2},u^{q^2}+v^q+v^{q^2})$ belongs to  $\langle P,Q\rangle_{\F_{q^4}}$ if and only if the rank of the following matrix is $2$: 
$$
M(P,Q,R):=\begin{pmatrix}
x&y&x^q+y^{q^2}&x^{q^2}+y^q+y^{q^2}\\
z&t&z^q+t^{q^2}&z^{q^2}+t^q+t^{q^2}\\
u&v&u^q+v^{q^2}&u^{q^2}+v^q+v^{q^2}\\
\end{pmatrix}.
$$
If $xt+zy\neq 0$ then 
$$\begin{cases}
\frac{(x^q+y^{q^2})t+(z^q+t^{q^2})y}{xt+zy}u +\frac{(x^q+y^{q^2})z+(z^q+t^{q^2})x}{xt+zy}v+u^q+v^{q^2}=0,\\
\frac{(x^{q^2}+y^q+y^{q^2})t+(z^{q^2}+t^q+t^{q^2})y}{xt+zy}u +\frac{(x^{q^2}+y^q+y^{q^2})z+(z^{q^2}+t^q+t^{q^2})x}{xt+zy}v+u^{q^2}+v^q+v^{q^2}=0.\\
\end{cases}
$$
By Propositions \ref{Prop:caso1} and \ref{Prop:caso2} the system above has at most $q^3$ solutions $(u,v)\in \mathbb{F}_{q^4}^2$.

Suppose now $xt+zy= 0$ then 
$$\begin{cases}
A_{2,3}u+A_{1,3} v=0\\
A_{2,4}u+A_{1,4} v=0\\
A_{3,4}v+A_{2,4}(u^q+v^q)+A_{2,3}(u^{q^2}+v^{q^2}+v^q)=0\\
A_{3,4}u+A_{1,4}(u^q+v^q)+A_{1,3}(u^{q^2}+v^{q^2}+v^q)=0,
\end{cases}$$
where
\begin{eqnarray*}
A_{2,3}&=&(x^q+y^{q^2})t+(z^q+t^{q^2})y,\\
A_{2,4}&=&(x^{q^2}+y^q+y^{q^2})t+(z^{q^2}+t^q+t^{q^2})y,\\
A_{3,4}&=&(x^{q^2}+y^q+y^{q^2})(z^q+t^{q^2})+(x^q+y^{q^2})(z^{q^2}+t^q+t^{q^2}),\\
A_{1,4}&=&(x^{q^2}+y^q+y^{q^2})z+x(z^{q^2}+t^q+t^{q^2}),\\
A_{1,3}&=&(x^q+y^{q^2})z+(z^q+t^{q^2})x.
\end{eqnarray*}
Since the rank of 
$$
\begin{pmatrix}
x&y&x^q+y^{q^2}&x^{q^2}+y^q+y^{q^2}\\
z&t&z^q+t^{q^2}&z^{q^2}+t^q+t^{q^2}\\
\end{pmatrix}$$
is $2$, $(A_{1,3},A_{2,3},A_{3,4})\neq (0,0,0)$ otherwise $x=y^q$ and $z=t^q$ and from $xt+yz=0$ one gets $y/t=x/z\in \mathbb{F}_q^*$, a contradiction. The claim follows from Proposition \ref{Prop:caso3}.
\end{proof}

\subsection{The duality}
Let $\Tr_{{q^4}/q}$ denote the trace function of $\F_{q^4}$ over
$\F_q$. The map $\Tr_{{q^4}/q}(X_0X_3-X_1X_2)$ defines a quadratic form of
$\F_{q^4}^4$ (regarded as $\F_q$-vector space) over $\F_q$. The polar form
associated with such a quadratic form is
$\sigma':=(\Tr_{{q^4}/q}\circ \sigma) : \F_{q^4}^4\times \F_{q^4}^4\longrightarrow \fq$, where $$\sigma\left((X_0,X_1,X_2,X_3),(Y_0,Y_1,Y_2,Y_3)\right)=X_0Y_3+X_3Y_0-X_1Y_2-X_2Y_1.$$ 

If $f$ is an $\F_q$--linear map from $\F_{q^4}$ to
$\F_{q^4}$ we will denote by $f^\top$ the \textbf{adjoint} of $f$ with
respect to the bilinear form of $\F_{q^4}$ (treated as
$\F_q$-vector space) $\Tr_{q^4/q}(xy)$. In other words, the adjoint $f^\top$ of $f$ is defined by the rule: 
$\Tr_{q^4/q}(xf(y))=\Tr_{q^4/q}(yf^\top(x))$ for any $x,y \in \F_{q^4}$.

Let $h_1,h_2,g_1,g_2$ be $\F_q$-linear maps over $\F_{q^4}$, and let \[X=\left\{\left(x,y,h_1(x)+h_2(y),g_1(x)+g_2(y)\right)\colon x,y\in\F_{q^4}\right\}\] be an $\F_{q}$-subspace of $\F_{q^4}^4$ of dimension $8$. Straightforward computations show that the orthogonal complement of $X$ with respect to the
$\F_q$-bilinear form $\sigma'$ is 
\[X^{\tau'}=\left\{\left(z,t,-g_2^\top(z)-h_2^\top(t),-h_1^\top(t)-g_1^\top(z)\right)\colon z,t\in\F_{q^4}\right\}.\]

Hence, for $q=2^h$, $h\equiv 1\pmod 2$, the orthogonal complement of $U$  is the $[8,4]_{q^4/q}$ system
\begin{equation}\label{formUperp}
U^{\tau'}= \left\{\left(z,t,z^{q^{3}}+z^{q^{2}}+t^{q^{2}},z^{q^{2}}+t^{q^{3}}\right) : z,t \in \mathbb{F}_{q^4}^2 \right\}.
\end{equation}

\begin{prop}\label{equivperp}
The $\F_q$-subspaces $U$ and $U^{\tau'}$ are $\mathrm{GL}(4,q^4)$-equivalent.
\end{prop}

\begin{proof}
Straightforward computations show that 
\[
\begin{pmatrix}
0 & 1 & 1 & 0\\
1 & 0 & 1 & 1\\
0 & 1 & 0 & 1\\
1 & 0 & 1 & 0
\end{pmatrix}
\begin{pmatrix}
x\\
y\\
x^q+y^{q^2}\\
x^{q^2}+y^q+y^{q^2}
\end{pmatrix}=
\begin{pmatrix}
z\\
t\\
z^{q^{3}}+z^{q^{2}}+t^{q^{2}}\\
z^{q^{2}}+t^{q^{3}}
\end{pmatrix},
\]
putting
$$ z= x^q+y+y^{q^2},\qquad t=x+x^q+x^{q^2}+y^q.$$
\end{proof}

\begin{cor}\label{cor:equivalentTracedual}
 For any $\mC\in \Psi([U])$ and $\cD\in\Psi([U^{\tau'}])$, we have that $\mC$ and $\cD$ are equivalent. In other words, $$\Psi([U])=\Psi([U^{\tau'}]).$$
\end{cor}

\begin{remark}
 
In \cite[Section 3]{CsMPZ2019}, another type of duality has been introduced. 	
	Let $U$ be an $n$-dimensional $\F_q$-subspace of a vector space $V=V(k,q^m)$, with $n>k$. By \cite[Theorems 1, 2]{LuPo2004} (see also \cite[Theorem 1]{LuPoPo2002}), there is an embedding of $V$ in $Z=V(n,q^m)$ with $Z=V \oplus \Gamma$ for some $(n-k)$-dimensional $\F_{q^m}$-subspace $\Gamma$ such that
	$U=\langle W,\Gamma\rangle_{\F_{q}}\cap V$, where $W$ is a $n$-dimensional $\F_q$-subspace of $Z$, $\langle W\rangle_{\F_{q^m}}=Z$ and $\Gamma\cap V=W\cap \Gamma=\{{ 0}\}$.
	Then the quotient space $Z/\Gamma$ is isomorphic to $V$ and under this isomorphism $U$ is the image of the $\F_q$-subspace $W+\Gamma$ of $Z /\Gamma$.
	Now, let $\beta'\colon W\times W\rightarrow\F_{q}$ be a non-degenerate bilinear form on $W$. Then $\beta'$ can be extended to a non-degenerate bilinear form $\beta\colon Z\times Z\rightarrow\F_{q^m}$.
	Let $\perp$ and $\perp'$ be the orthogonal complement maps defined by $\beta$ and $\beta'$ on the lattice of $\F_{q^m}$-subspaces of $Z$ and of $\F_q$-subspaces of $W$, respectively.
	The $k$-dimensional $\F_q$-subspace $W+\Gamma^{\perp}$ of the quotient space $Z/\Gamma^{\perp}$  will be denoted by $\bar U$ and we call it the \textbf{Delsarte dual}{} of $U$ with respect to $\beta'$. By \cite[Remark 3.7]{CsMPZ2019}, up to $\GL(n,q)$-equivalence,  the Delsarte dual of an $n$-dimensional $\F_q$-subspace does not depend on the choice of the non-degenerate bilinear form on $W$. 
	
	\medskip

	Let $U$ be the $(2,3)_q$-evasive $\F_q$-subspace of $V=\F_{q^4}^4$ defined in \eqref{formU}. Using the notations above we can embed $V$ in $Z=\F_{q^4}^8$ in such a way that
	\begin{align*}
	    &V=\left\{\left(Y_0,Y_1,Y_2,0,Y_4,0,0,0\right): Y_0,Y_1,Y_2,Y_4\in\F_{q^4}\right\},\\
	    &W=\left\{(x,y,x^q,y^q,x^{q^2},y^{q^2},x^{q^3},y^{q^3}): x,y\in\F_{q^4}\right\},\\
	    &\Gamma=\left\{\left(0,0,X_2,X_3,X_2-X_3,-X_2,X_6,X_7\right): X_2,X_3,X_6,X_7\in\F_{q^4}\right\}.
	\end{align*}
	Straightforward computations show that $U=\langle W,\Gamma\rangle_{\F_{q}}\cap V$.

	Let $({\bf x},{\bf y}):=(x,y,x^q,y^q,x^{q^2},y^{q^2},x^{q^3},y^{q^3}), ({\bf u},{\bf v}):=(u,v,u^q,v^q,u^{q^2},v^{q^2},u^{q^3},v^{q^3})$ and  consider the bilinear form $\beta'$ on $W$ defined as
	$$\beta'(({\bf x},{\bf y}),({\bf u},{\bf v}))=\Tr_{q^4/q}(xv-uy).$$
	Then 
	$$\Gamma^\perp=\left\{\left(Z_0,Z_1,Z_2,Z_3,Z_2-Z_3,-Z_2,0,0\right): Z_0,Z_1,Z_2,Z_3\in\F_{q^4}\right\}.$$
	Hence, the Delsarte dual $\bar U=\langle W,\Gamma^\perp\rangle_{\F_{q}}\cap \Delta$, where $$\Delta=\left\{\left(0,0,X_2,X_3,0,0,X_6,X_7\right): X_2,X_3,X_6,X_7\in\F_{q^4}\right\},$$ turns out to be
	\begin{align*}
	  \bar U&=\left\{(0,0,x^q+y^{q^2},x^{q^2}+y^q+y^{q^2},0,0,x^{q^3},y^{q^3}): x,y\in\F_{q^4}\right\}\\
	  &=\left\{(0,0,t^{q^3}+z^{q^2},t^{q^2}+z^{q^3}+z^{q^2},0,0,t,z): x,y\in\F_{q^4}\right\},
	\end{align*}
    which is equivalent to $U$ (cf. \eqref{formUperp} and Proposition \ref{equivperp}).

\end{remark}

\subsection{The associated MRD code}

From a coding theoretic point of view, the existence of the $[8,4]_{q^4/q}$ system $U$ opens new perspectives and shows the existence of MRD codes ``which are better than others''. The last sentence is in quotes because there are several criteria for which a code might be thought to be better than another, and our purpose is not to underline a better performance in every aspect. What we mean with that sentence is that the generalized rank weights of this code are actually larger than the ones of the {already known }MRD codes with the same parameters. 

For the rest of this section we fix an $[8,4]_{q^4/q}$ code $\mC$ associated to $U$, and study its parameters, comparing them to those of the known constructions of MRD codes and to the best possible parameters that an $[8,4]_{q^4/q}$ code can have. Concretely, we can take $\mC$ to be the $[8,4]_{q^4/q}$ code with generator matrix 
$$\begin{pmatrix}
\alpha_1 & \alpha_2 & \alpha_3 & \alpha_4 & 0 & 0 & 0 & 0 \\
0 & 0 & 0 & 0 & \beta_1 & \beta_2 & \beta_3 & \beta_4 \\
\alpha_1^q & \alpha_2^q & \alpha_3^q & \alpha_4^q & \beta_1^{q^2} & \beta_2^{q^2} & \beta_3^{q^2} & \beta_4^{q^2} \\
\alpha_1^{q^2} & \alpha_2^{q^2} & \alpha_3^{q^2} & \alpha_4^{q^2} & \beta_1^q+\beta_1^{q^2} & \beta_2^q+\beta_2^{q^2} & \beta_3^q+\beta_3^{q^2} & \beta_4^q+\beta_4^{q^2} 
\end{pmatrix}\in \F_{q^4}^{4\times8},$$
where $(\alpha_1,\alpha_2,\alpha_3,\alpha_4)$ and $(\beta_1,\beta_2,\beta_3,\beta_4)$ are both $\fq$-bases of $\F_{q^4}$.

The following is just a rewriting of Theorem \ref{thm:Uscattered} and Theorem \ref{thm:Uevasive}, together with the fact that codes associated to scattered $[\frac{km}{2},k]_{q^m/q}$ systems are MRD, i.e. their parameters meet \eqref{eq:singleton} with equality; see \cite[Theorem 3.2]{csajbok2017maximum}.

\begin{prop}\label{prop:Cgenweights}
 Let $q=2^h$, with $h \equiv 1 \pmod 2$. Then $\mC$ is an $[8,4,(3,5,7,8)]_{q^4/q}$ code. In other words, $\mC$ is an $[8,4]_{q^4/q}$ MRD code, whose generalized rank weights are 
 $$ \dd_{\rk, 1}(\mC)=3, \quad \dd_{\rk,2}(\mC)=5, \quad \dd_{\rk,3}(\mC)=7, \quad \dd_{\rk,4}(\mC)=8. $$
\end{prop}

\begin{proof}
Recall that, from Theorem \ref{thm:charact_evasive_genweights}, if $\mC$ is an $\Fmk$ code, then $\dd_{\rk,k-h}(\mC) = n-r$ if and only if $\mU$ is $(h,r)_q$-evasive but not $(h,r-1)_q$-evasive. Applying this theorem with $k=4$ and $n=8$, since $U$ is scattered, $U$ is $(1,1)_q$-evasive and $(3,5)_q$-evasive \cite[Theorem 4.2]{BL2000}, and hence $\dd_{\rk,1}(\mC) = 3$ and $\dd_{\rk,3}(\mC) = 7$. Since $U$ is $(2,3)_q$-evasive but it is not  $2$-scattered,  $\dd_{\rk,2}(\mC) = 5$.

\end{proof}

 By the Singleton bound  \eqref{eq:singleton}, the minimum rank distance, that is the first generalized rank weight of $\mC$, attains the maximum possible value. In other words, $\mC$ is MRD. This is not a big news, since MRD codes with the same parameters were already known, and constructed as the direct sum of two $[4,2]_{q^4/q}$ MRD codes (e.g. Delsarte-Gabidulin codes). However, codes obtained as direct sum of MRD codes have generalized rank weights different from those of $\mC$. This is explained by the following general result.

\begin{prop}\label{prop:genweights_directsum}
 Let ${\mC}_1,{\mC}_2$ be two $[n,k,n-k+1]_{q^m/q}$ MRD codes. Then,  ${\mC}_1 \oplus {\mC}_2$ is a $[2n,2k]_{q^m/q}$ code whose  generalized rank weights are given by 
 $$ \dd_{\rk,i}({\mC}_1\oplus{\mC}_2)=\begin{cases} n-k+i & \mbox{ if } 1 \leq i \leq k, \\
 2n-2k+i & \mbox{ if } k+1 \leq i \leq 2k.\end{cases}$$
\end{prop}

\begin{proof}
Let us divide the proof in four steps. 

\begin{enumerate}
    \item First of all, by the monotonicity of Proposition \ref{prop:gen_weights_properties}(1), since $\mC_1$ and $\mC_2$ are MRD codes of length $n\leq m$, we must have
$$\dd_{\rk,i}(\mC_j)=n-k+i, \qquad \mbox{ for each } i \in \{1,\ldots,k\}, j \in\{1,2\}.$$
In particular, $\mC_1$ and $\mC_2$ are $[n,k,(n-k+1,n-k+2,\ldots,n)]_{q^m/q}$ codes. 

\item  Let us write $\cD:=\mC_1\oplus \mC_2$ and let 
$$d_i:=\dd_{\rk,i}(\cD),\quad \mbox{ for each } i \in \{1,\ldots, 2k\}.$$ Since every codeword in $\cD$ is of the form $(u \, \mid \, v)$, with $u\in \mC_1$, $v\in \mC_2$,  it is clear that the minimum rank distance of $\cD$ -- which is the first generalized rank weight -- coincides with the minimum among the minimum rank distances of $\mC_1$ and $\mC_2$, i.e.
 $$d_1=n-k+1.$$
 Now, let $V=V(k,q^m)$ and let $U_1, U_2$ be two $[n,k,(n-k+1,n-k+2,\ldots,n)]_{q^m/q}$ systems associated to $\mC_1$ and $\mC_2$, respectively. The $[2n,2k]_{q^m/q}$ system $U_1\oplus U_2 \subseteq V\oplus V=V(2k,q^m)$ is then associated to the code $\cD$. Now, for each $i \in \{2,\ldots,k\}$, consider an $\Fm$-subspace of codimension $i$ in $V\oplus V$, given by $\Pi_i=V\oplus H_i$, where $H_i$ is an $\Fm$-subspace of codimension $i$ in $V$ such that 
 $$\dd_{\rk,i}(\mC_2)=n-\dim_{\fq}(U_2\cap H_i)=n-k+i.$$ Then, 
 for each $i \in \{2,\ldots,k\}$ we have  
 \begin{align*} d_i &= 2n-\max\{\dim_{\fq}((U_1\oplus U_2)\cap \Pi) \,:\, \Pi \subseteq V\oplus V \mbox{ with } \dim_{\Fqm}(\Pi)=2k-i \}  \\
 &\leq 2n-\dim_{\fq}((U_1\oplus U_2)\cap \Pi_i)\\
 &=2n-\dim(U_1)-\dim(U_2\cap H_i)=n-k+i. \end{align*}
 On the other hand, since $d_1=n-k+1$, by the monotonicity of Proposition \ref{prop:gen_weights_properties}(1) we have $d_i\geq n-k+i$ for each $i \in \{2,\ldots,k\}$, and thus equality.
 
 \item Let us now consider the dual code $\cD^\perp$, which is a $[2n,2n-2k,(d_1^\perp,\ldots, d_{2n-2k}^\perp)]_{q^m/q}$ code. It is easy to see that $\cD^\perp=\mC_1^\perp \oplus \mC_2^\perp$. However, since the dual of an MRD is itself MRD,  by step (a)
 $\mC_1^\perp$ and $\mC_2^\perp$ are both $[n,n-k,(k+1,\ldots,n)]_{q^m/q}$ codes, and by step (b) the first $n-k$ generalized rank weights of $\cD^\perp$ are 
 $$d_i^\perp=k+i, \quad \mbox{ for each } i \in \{1,\ldots, n-k\}.$$
 
 \item By the Wei-type duality of Proposition \ref{prop:gen_weights_properties}(2),  we  have 
 $$ \{d_{k+1},\ldots,d_{2k} \}\cup \{2n+1-d_{n-k+1}^\perp,\ldots,2n+1-d_{2n-2k}^\perp \}=\{1,\ldots, n-k\}\cup \{2n-k+1,\ldots,2n\}, $$
 which forces, by the monotonicity of Proposition \ref{prop:gen_weights_properties}(1), to have 
 $$ d_{k+i}=2n-k+i, \quad \mbox{ for each } i \in \{1,\ldots,k\},$$
 concluding the proof.
  
 \end{enumerate}
\end{proof}

The following result immediately follows from Proposition \ref{prop:genweights_directsum}.

\begin{cor}\label{cor:directsum84}
If $\mC_1,\mC_2$ are two  $[4,2]_{q^4/q}$ MRD codes, then ${\mC}_1\oplus {\mC}_2$ is an $[8,4,(3,4,7,8)]_{q^4/q}$ code.  In other words, ${\mC}_1\oplus {\mC}_2$ is an $[8,4]_{q^4/q}$ MRD code, whose generalized rank weights are 
 $$ \dd_{\rk, 1}(\mC)=3, \quad \dd_{\rk,2}(\mC)=4, \quad \dd_{\rk,3}(\mC)=7, \quad \dd_{\rk,4}(\mC)=8. $$
\end{cor}

Let us fix now $\mathcal D_1,\mathcal D_2$ to be two arbitrary $[4,2]_{q^4/q}$ MRD codes. Hence, comparing the results of Proposition \ref{prop:Cgenweights} and Corollary \ref{cor:directsum84}, we have that first, third, and fourth generalized rank weights of the  codes $\mC$ and $\mathcal D_1\oplus\mathcal D_2$ coincide, but 
$$ \dd_{\rk,2}(\mC)=5>4=\dd_{\rk,2}(\mathcal D_1\oplus\mathcal D_2),$$
explaining finally what we meant with the claim that the MRD code $\mC$ is ``better'' than the MRD code $\mathcal D_1\oplus\mathcal D_2$. 

The second difference that we want to underline between $\mC$ and $\mathcal D_1\oplus\mathcal D_2$ concerns their minimality. Indeed, we obtain the following.

\begin{cor}
 The code $\mC$ is minimal, while the code $\mathcal D_1\oplus\mathcal D_2$ is not.
\end{cor}

\begin{proof}
 It immediately follows from Theorem \ref{thm:minimal_secondweight}.
\end{proof}

\begin{table}[!ht]
\begin{center}
\begin{tabular}{c|c|c|c|c|c|c}
    Code  & is MRD? & $\dd_{\rk,1}$ & $\dd_{\rk,2}$ & $\dd_{\rk,3}$ & $\dd_{\rk,4}$ & is minimal? \\
     \hline 
     $\mC$ & yes & $3$ & $5$ & $7$ & $8$ & yes \\
          $\mathcal D_1\oplus \mathcal D_2$ & yes & $3$ & $4$ & $7$ & $8$ & no \\
\end{tabular}
\caption{The table recaps the properties and the generalized rank weights of the code $\mC$ compared to those of  $[8,4]_{q^4/q}$ MRD codes obtained as direct sum of two $[4,2]_{q^4/q}$ MRD codes.}\label{fig}
\end{center}

\end{table}

\begin{remark}
Apart from being better than  those of {the type} $\cD_1\oplus\cD_2$, the generalized rank weights of $\mC$ are actually the largest possible ones for the given parameters $(n,k,m)=(8,4,4)$. Indeed, by Proposition \ref{prop:bounds_genweights}, it is immediate to see that $\dd_{\rk,1}(\mC)$, $\dd_{\rk,3}(\mC)$, and $\dd_{\rk,4}(\mC)$  are the largest possible values. Furthermore, if there exists an $[8,4]_{q^4/4}$ code $\cD$ such that $\dd_{\rk,2}(\cD)\geq 6${, then Theorem \ref{thm:charact_evasive_genweights} yields the $2$-scatteredness of the $[8,4]_{q^4/4}$ system $W$ associated to $\cD$.} However this is not possible due to the bound on $h$-scattered $[n,k]_{q^m/q}$ systems given by
$$n \leq \frac{km}{h+1}$$
proved in \cite[Theorem 2.3]{CsMPZ2019}, which in this case would read as {$8\leq \frac{16}{3}$}, a contradiction.
\end{remark}

 One can see   by Proposition \ref{prop:gen_weights_properties}(2) that the dual code $\mC^\perp$ is also an $[8,4,(3,5,7,8)]_{q^4/q}$ code. However, we can say more about the dual of $\mC$. 
 
 \begin{prop}
  The dual code $\mC^\perp$ is equivalent to $\mC$.
 \end{prop}
 
 \begin{proof}
  Let us fix a normal $\fq$-basis  ${\underline{\alpha}}:=(\alpha, \alpha^{q}, \alpha^{q^2},\alpha^{q^3})$ of $\F_{q^4}$, and let $\underline{\gamma}:=(\gamma,\gamma^q,\gamma^{q^2},\gamma^{q^3})$ be its dual basis with respect to the trace. A generator matrix for (a code in the equivalence class of) $\mC$ is given by
  $$G=\begin{pmatrix} \underline{\alpha} & 0 \\ 0 & \underline{\alpha} \\
   \underline{\alpha}^q & \underline{\alpha}^{q^2} \\ \underline{\alpha}^{q^2} & \underline{\alpha}^q+\underline{\alpha}^{q^2} \end{pmatrix} \in \F_{q^4}^{4 \times 8}, $$
 where $\underline{\alpha}^{q^i}$ denotes the vector obtained from $\underline{\alpha}$ by raising each entry to the $({q^i})$-th power. Straightforward computations show that the dual of $\mC$ is generated by the matrix
   $$H=\begin{pmatrix}  0 & \underline{\gamma}^{q^3} \\ \underline{\gamma}^{q^3} & 0 \\
    \underline{\gamma}^{q} & \underline{\gamma}^{q}+\underline{\gamma}^{q^2}  \\ \underline{\gamma}^{q^2} & \underline{\gamma}^{q} \end{pmatrix}\in \F_{q^4}^{4 \times 8}.$$
    From this, it is immediate to see that the (equivalence class of the) $q$-system associated to $\mC^\perp$ is $U^{\tau'}$ defined in \eqref{formUperp}. The claim now follows from Corollary \ref{cor:equivalentTracedual}.
 \end{proof}

\section{Conclusion and open problems}\label{sec:conclusions}

In this work we analyzed short minimal rank-metric codes, or, equivalently, small linear cutting blocking sets. We managed to characterize them in terms of the second generalized weight{, showing that } a rank-metric code  is minimal if and only if its second generalized rank weight is larger than the underlying field extension degree (Theorem \ref{thm:minimal_secondweight}). This result was proved using the evasiveness properties of the associated $q$-system that characterize a linear cutting blocking set (Theorem \ref{Thm:caract}). Hence, motivated by finding small linear cutting blocking sets, we {provide } a construction of the smallest one in $V(4,q^4)$, when $q$ is an odd power of $2$ (Theorem \ref{thm:Uevasive}). As a byproduct, this construction produces a minimal $[8,4]_{q^4/q}$ rank-metric code which is also MRD and it is  the first example of an $[8,4]_{q^4/q}$ MRD code that it is not obtained as a direct sum of two $[4,2]_{q^4/q}$ MRD codes, and hence it provides a genuinely new MRD code. In addition, the parameters are also different, as  we show that the generalized rank weights of this new MRD code are larger than those of the previously known MRD codes; see Table \ref{fig} for all the comparisons. This result opens new concrete research directions, since it tells us that there are MRD codes which are better than others. However, first one has to quantify how much better one can do, finding better bounds for the generalized rank weight of a code.

\medskip
\noindent\textbf{Open Problem 1.} Find new bounds for the generalized rank weights of an $\Fmk$ code, improving on the bounds of Proposition \ref{prop:bounds_genweights}.
\medskip

Among the properties of the $q$-system $U$ and of its associated codes, an important feature that gives hope to generalize it for larger ambient spaces is its extremely compact and elementary description. Thus, it would be nice to determine if it might be a special case of a more general pattern. However, this can be done in two directions.

\medskip 
\noindent\textbf{Open Problem 2.} Generalize the construction of the $q$-system $U$ of Section \ref{sec:construction_evasive} in order to obtain more general MRD codes with higher generalized weights.

\medskip
\noindent\textbf{Open Problem 3.} Generalize the construction of the $q$-system $U$ of Section \ref{sec:construction_evasive} in order to obtain more general short minimal rank-metric codes.  
\medskip

\section*{Acknowledgements} 

This research was supported by the Italian National Group for Algebraic and Geometric Structures and their Applications (GNSAGA - INdAM).

\end{document}